\documentclass[12pt]{amsart}
\usepackage{color,array, amssymb, amscd,slashed}
\usepackage{graphicx}
\usepackage[normalem]{ulem}
\newtheorem{Theorem}{Theorem}[section]
\newtheorem{Lemma}[Theorem]{Lemma}
\newtheorem{Proposition}[Theorem]{Proposition}

\theoremstyle{definition}
\newtheorem{Definition}[Theorem]{Definition}
\theoremstyle{remark}

\newtheorem{Remark}[Theorem]{Remark} 

\numberwithin{equation}{section}
\newcommand{\R}{\mathbb R}
\newcommand{\C}{\mathbb C^{\prime}}

\newcommand{\D}{\mathbb D}

\newcommand{\Q}{\mathbb Q}
\newcommand{\Min}{\mathbb L^3}

\newcommand{\SLR}{\mathrm{SL}_2 \mathbb R}
\newcommand{\LSLR}{\Lambda {\rm SL}_{2} \R_{\sigma}}
\newcommand{\LSLRP}{\Lambda^+ {\rm SL}_{2} \R_{\sigma}}
\newcommand{\LSLRN}{\Lambda^- {\rm SL}_{2} \R_{\sigma}}

\newcommand{\ISU}{{\rm SU}_{1, 1}^{\prime}}
\newcommand{\LISU}{\Lambda^{\prime} {\rm SU}_{1, 1 \sigma}^{\prime}}

\newcommand{\SLC}{{\rm SL}_2 \mathbb C^{\prime}}
\newcommand{\LSLC}{\Lambda^{\prime} {\rm SL}_{2} \C_{\sigma}}
\newcommand{\LSLCP}{\Lambda^{\prime +} {\rm SL}_{2} \C_{\sigma}}
\newcommand{\LSLCN}{\Lambda^{\prime -} {\rm SL}_{2} \C_{\sigma}}
\newcommand{\LSLCPN}{\Lambda^{\prime +}_* {\rm SL}_{2} \C_{\sigma}}
\newcommand{\LSLCNN}{\Lambda^{\prime -}_* {\rm SL}_{2} \C_{\sigma}}

\newcommand{\Uone}{{\rm U}_1^{\prime}}
\newcommand{\id}{\operatorname{id}}

\newcommand{\isu}{\mathfrak{su}_{1, 1}^{\prime}}
\newcommand{\lisu}{\Lambda^{\prime} \mathfrak{su}^{\prime}_{1, 1 \sigma}}

\newcommand{\slc}{\mathfrak{sl}_2 \mathbb C^{\prime}}
\newcommand{\lslc}{\Lambda^{\prime} \mathfrak{sl}_{2} \mathbb C^{\prime}_{\sigma}}
\newcommand{\lslcn}{\Lambda^{\prime-} \mathfrak{sl}_{2} \mathbb C^{\prime}_{\sigma}}
\newcommand{\lslcp}{\Lambda^{\prime+} \mathfrak{sl}_{2} \mathbb C^{\prime}_{\sigma}}

\newcommand{\slR}{\mathfrak{sl}_2 \mathbb R}
\newcommand{\lslR}{\Lambda \mathfrak{sl}_2 \mathbb R_{\sigma}}

\newcommand{\ad}{\operatorname{Ad}}
\newcommand{\di}{\operatorname{diag}}
\newcommand{\Nil}{{\rm Nil}_3}

\renewcommand{\Re}{\operatorname {Re}}
\renewcommand{\Im}{\operatorname {Im}}
\renewcommand{\l}{\lambda}
\renewcommand{\S}{\mathbb S}

\newcommand{\ip}{i^{\prime}}

\usepackage{amsmath}	
\begin{document}
\title{Timelike minimal surfaces in the three-dimensional Heisenberg group}
 \author[H.~Kiyohara]{Hirotaka Kiyohara}
 \address{Department of Mathematics, Hokkaido University, 
 Sapporo, 060-0810, Japan}
 \email{kiyosannu@eis.hokudai.ac.jp}
\thanks{The first named author is supported by JST SPRING, Grant Number JPMJSP2119.}

 \author[S.-P.~Kobayashi]{Shimpei Kobayashi}
 \address{Department of Mathematics, Hokkaido University, 
 Sapporo, 060-0810, Japan}
 \email{shimpei@math.sci.hokudai.ac.jp}
 \thanks{The second named author is partially supported by Kakenhi 18K03265.}
 \date{\today}
\pagestyle{plain}
\begin{abstract}
 Timelike surfaces in the three-dimensional Heisenberg group with left invariant 
 semi-Riemannian metric are studied. In particular, non-vertical timelike minimal 
 surfaces are characterized by the non-conformal Lorentz harmonic maps into 
 the de Sitter two sphere. On the basis of the characterization, 
 the generalized Weierstrass type representation 
 will be established through the loop group decompositions.
\end{abstract}
\maketitle    
\section{Introduction}
 Constant mean curvature surfaces in three-dimensional homogeneous spaces, 
 specifically   Thurston's eight model spaces \cite{Thurston},
 have been intensively studied in recent years.
 One of the reasons is a seminal paper by Abresch-Rosenberg \cite{Abresch-Rosenberg:Acta},
 where they introduced a quadratic differential, the so-called the \textit{Abresch-Rosenberg differential}, analogous to the Hopf differential for surfaces in the space forms 
 and showed that it was holomorphic for a constant 
 mean curvature surface in various classes of three-dimensional homogeneous spaces,
 such as the Heisenberg group $\Nil$, the product spaces $\S^2 \times \R$
 and $\mathbb H^2 \times \R$ etc, see \cite{Abresch-Rosenberg} in detail.
 It is evident that holomorphic quadratic differentials are fundamental 
 for study of global geometry of surfaces, \cite{Fer-Mira2}.
 On the one hand, Berdinsky-Taimanov developed integral representations  of surfaces in three-dimensional homogeneous spaces
 by using the generating spinors and the nonlinear Dirac type equations, 
 \cite{Ber:Heisenberg, BT:Sur-Lie}. They were natural generalizations of 
 the classical Kenmotsu-Weierstrass representation 
 for surfaces in the Euclidean three-space.

 Combining the Abresch-Rosenberg differential 
 and the nonlinear Dirac equation with generating spinors, in  \cite{DIKAsian, DIKtop}, 
 Dorfmeister, Inoguchi and the second named author of this paper have 
 established the loop group method for  minimal surfaces in $\Nil$, where the following
 left-invariant Riemannian metric has been 
 considered on $\Nil$: 
\[
 ds^2 = dx_1^2 + dx_2^2 +  \left(dx_3 + \frac12 (x_2 dx_1 - x_1 dx_2)\right)^2,
\]
 In particular, all non-vertical minimal surfaces in $\Nil$ have been constructed from 
 holomorphic data, which have been called the \textit{holomorphic potentials}, 
 through the loop group decomposition, the so-called Iwasawa decomposition,
 and the construction  has been commonly called the \textit{generalized Weierstrass type representation}.
 In this loop group method, the Lie group structure of $\Nil$ and 
 harmonicity of the left-translated normal Gauss map 
 of a non-vertical surface, which obviously took values in a hemisphere in 
 the Lie algebra of $\Nil$, 
 were essential tools.
 To be more precise, a surface in $\Nil$ is minimal if and only if 
 the left-translated normal Gauss map is a non-conformal harmonic map 
 with respect to the \textit{hyperbolic metric} on the hemisphere, that is, 
 one considers the hemisphere as the hyperbolic two space not the two sphere
 with standard metric.
 Since the hyperbolic two space is one of the standard symmetric spaces and the 
 loop group method of 
 harmonic maps from a Riemann surface into a symmetric space have been 
 developed very well \cite{DPW}, thus we have obtained 
 the generalized Weierstrass type  representation.

 On the one had, it is easy to see 
 that the three-dimensional Heisenberg group $\Nil$ can have the following 
 left-invariant \textit{semi-Riemannian} metrics:
\[
 ds^2_{\pm} = \pm dx_1^2 + dx_2^2 \mp  \left(dx_3 + \frac12 (x_2 dx_1 - x_1 dx_2)\right)^2.
\]
 Moreover in \cite{Rah}, it has been shown that 
 the left-invariant semi-Riemannian metrics on $\Nil$
 with $4$-dimensional isometry group only are
 the metrics $ds^2_{\mp}$.
 Therefore a natural problem is study of spacelike/timelike,   
 minimal/maximal surfaces in $\Nil$ with the above semi-Riemannian 
 metrics in terms of the generalized Weierstrass type representations.
%

 In this paper we will consider 
 timelike surfaces in $\Nil$ with the semi-Riemannian metric $ds^2_{-}$.
 For defining the Abresch-Rosenberg differential and the nonlinear Dirac equations 
 with generating spinors, the \textit{para-complex structure} on a timelike 
 surface is essential, and we will systematically develop theory 
 of timelike surfaces using the para-complex structure, the Abresch-Rosenberg 
 differential and the nonlinear Dirac equations with generating spinors in 
 Section \ref{sc:timelikesurf}.
 Then the first of the main results in this paper is Theorem \ref{thm:main}, where 
 non-vertical 
 timelike minimal surfaces in $\Nil$ will be characterized in terms of harmonicity 
 of the left-translated normal Gauss map.
 To be more precise, the left-translated normal Gauss map of a timelike surface 
 takes values in 
 the lower half part of the de Sitter two sphere $\widetilde{\mathbb S}^2_{1-} 
 =\{(x_1, x_2, x_3) \in \mathfrak{nil}_3 = \Min \mid 
  - x_1^2 + x_2^2 + x_3^2=1, x_3< 0\}$, but it is not a Lorentz harmonic map into 
 $\widetilde{\mathbb S}^2_{1-}$ with respect to the standard metric on the de Sitter sphere. It will be shown that  by combining two stereographic projections, the left-translated normal Gauss map can take values in  
  the upper half part of the de Sitter 
 two sphere with interchanging $x_1$ and $x_2$, that is, $\S^2_{1+} = \{(x_1, x_2, x_3) \in \Min \mid 
 x_1^2 -x_2^2 + x_3^2=1, x_3>0\}$, see Figure \ref{fig:desitter},
 and it is
 a non-conformal Lorentz harmonic into $\S^2_{1+}$ if and only if the timelike surface is minimal, see Section \ref{gaussmap}  in details.
Note that timelike minimal surfaces in $(\Nil, ds^2_-)$
 have been studied through the Weierstrass-Enneper type representation 
 and the Bj\"oring problem in \cite{IKOS, LMM, CDM, CMO:Bjorling, SSP, CMO:Minimal,  CO, Magid, Kondreak}.

 It has been known that timelike constant mean curvature surfaces in 
 the three-dimensional Minkowski space $\Min$ could be characterized by 
 a Lorentz harmonic map into the de Sitter two space, \cite{Inoguchi, IT, DIT, BS:CMC, 
 BS:Cauchy}. In fact
 the Lorentz harmonicity of the unit normal of a timelike surface in $\Min$
 is equivalent to constancy of the mean curvature. 
 Furthermore, the generalized Weierstrass type representation for 
 timelike non-zero constant mean curvature surfaces has been established in 
 \cite{DIT}. In Theorem \ref{thm:Sym1}, we will show 
 that two maps, which are given by the logarithmic derivative of 
 one parameter family of moving frames of a non-conformal Lorentz 
 harmonic map (the so-called \textit{extended frame}) 
 into $\S^2_{1+}$ with respect to an additional parameter 
 (the so-called \textit{spectral parameter}), 
 define a timelike non-zero constant mean curvature 
 surface in $\Min$ and a non-vertical timelike minimal surface in $\Nil$,
 respectively. 

 From the view point of the loop group 
 construction of Lorentz harmonic maps, the construction in \cite{DIT} 
 is sufficient, however, 
 it is not enough for our study of timelike minimal surfaces in $(\Nil, ds_-^2)$.
 As we have mentioned above, for defining 
 the Abresch-Rosenberg differential and the nonlinear Dirac equation 
 with generating spinors the para-complex structure is essential. 
 Note that the para-complex structure has been used for study of timelike surface
 \cite{W:Lorentz, L:Lorentz}.
 We can then show that 
 the Abresch-Rosenberg differential is para-holomorphic 
 if a timelike surface has constant mean curvature, 
 Theorem \ref{thm:constant}, which is analogous to the fundamental result 
 of Abresch-Rosenberg.

 As a by-product of utilizing the para-complex structure, it is easy to 
 compare our construction with minimal surface in $(\Nil, ds^2)$, where 
 the complex structure has been used, and moreover, 
 the generalized Weierstrass type representation can be understood in 
 a unified way, that is, the Weierstrass data is just a
 $2$ by $2$ matrix-valued para-holomorphic function 
 and a loop group decomposition of 
 the solution of a para-holomorphic differential equation gives the 
 extended frame of a non-conformal Lorentz harmonic map in $\S^2_{1 +}$, 
 Theorem \ref{thm:Weierstrass}. One of 
 difficulties is that one needs to have appropriate loop 
 group decompositions in the para-complex setting, that is, 
 Birkhoff and Iwasawa decompositions.
 In Theorem \ref{thm:BIdecomposition}, by identifying 
 the double loop groups of $\SLR$, that is $\LSLR \times \LSLR$ and
 the loop group of $\SLC$, 
 that is, $\LSLC$ (where $\C$ denotes the para-complex number) by a natural isomorphism,
 we will obtain such decompositions. 
 Finally in Section \ref{sc:Example}, several examples will be shown by our loop group construction. 
 In particular \textit{B-scroll type minimal surfaces} in $\Nil$ will be established 
 in Section \ref{sbsc:Bscroll}.
 In Appendix \ref{timelikemin}, we will
 discuss timelike constant mean curvature surfaces in $\Min$, and 
 in Appendix \ref{sc:nopara}, we will see 
 the correspondence between our construction and the construction without 
 the para-complex structure in \cite{DIT}.

\section{Timelike surfaces in $\Nil$}\label{sc:timelikesurf}
 In this section we will consider timelike surfaces in $\Nil$.
 In particular we will use the para-complex structure and 
 the nonlinear Dirac equation for timelike surfaces. Finally the Lax pair type 
 system for timelike surface will be shown.
\subsection{$\Nil$ with indefinite metrics}
 The Heisenberg group is a $3$-dimensional Lie group
\[
 \Nil(\tau) = (\R^3 (x_1,x_2,x_3), \cdot)
\]
 for $\tau \neq 0$ with the multiplication
\[
 (x_1,x_2,x_3) \cdot (y_1,y_2,y_3)
 =
 (x_1+y_1,x_2+y_2,x_3+y_3+\tau(x_1y_2-y_1x_2)).
\]
 The unit element of $\Nil(\tau)$ is $(0,0,0)$.
 The inverse element of $(x_1,x_2,x_3)$ is $(-x_1,-x_2,-x_3)$.
 The groups $\Nil(\tau)$ and $\Nil(\tau')$ are isomorphic if $\tau \tau' \neq 0$.
 The Lie algebra $\mathfrak{nil}_3$ of $\Nil(\tau)$ is $\R^3$ with the relations:
\[
 [e_1,e_2] = 2\tau e_3, \quad [e_2,e_3] = [e_3,e_1] =0
\]
with respect to the normal basis $e_1=(1,0,0),e_2=(0,1,0),e_3=(0,0,1)$.
 In this paper we consider the left invariant indefinite metric $ds_-^2$ for
 $\Nil$ as follows:
\begin{equation}\label{eq:indefinitemetric}
 ds_-^2  = -(dx_1)^2 + (dx_2)^2 + \omega_{\tau} \otimes \omega_{\tau},
\end{equation}
 where $\omega_{\tau} = dx_3 + \tau (x_2 dx_1 - x_1 dx_2)$.
 Moreover, we fix the real parameter $\tau$ as $\tau =1/2$
 for simplicity.
 The vector fields $E_k \, (k=1,2,3)$ defined by
 \[
 E_1 = \partial_{x_1} - \frac{x_2}2 \partial_{x_3},\quad
 E_2 = \partial_{x_2} + \frac{x_1}2 \partial_{x_3}\quad
 {\rm and}\quad
 E_3 = \partial_{x_3}
 \]
 are left invariant corresponding to $e_1,e_2,e_3$ and orthonormal to each other
 with the timelike vector $E_1$ with respect to the metric $ds_-^2$.
 The Levi-Civita connection $\nabla$ of $ds_-^2$ is given by
\[
\begin{matrix}\vspace{2pt}
\nabla_{E_1} E_1 = 0,              & \nabla_{E_1} E_2 = \frac12E_3,  &\nabla_{E_1} E_3 = -\frac12E_2,\\
\vspace{2pt}
\nabla_{E_2} E_1 = -\frac12E_3,& \nabla_{E_2} E_2 = 0,              &\nabla_{E_2} E_3 = -\frac12E_1,\\
\nabla_{E_3} E_1 = -\frac12E_2,& \nabla_{E_3} E_2 = -\frac12E_1,&\nabla_{E_3} E_3 = 0    .\\
\end{matrix}
\]
\subsection{Para-complex structure}
 Let $\C$ be a real algebra spanned by $1$ and $\ip$ with following multiplication:
\[ (\ip)^2=1, \quad 1\cdot \ip = \ip \cdot 1= \ip. \] 
 An element of the algebra $\C = \R 1\oplus \R \ip$ is called a 
 \textit{para-complex number}. 
 For a para-complex number $z$ we can uniquely express $z=x+y\ip$ with some $x, y \in \R$.
 Similar to complex numbers, the real part $\Re z$,
 the imaginary part $\Im z$
 and the conjugate $\bar z$ of $z$ are defined by
\[ \Re z = x, \quad \Im z = y \quad {\rm and} \quad
 \bar z = x - y\ip.
\]
 For a para-complex number $z=x+y\ip \in \C$
 there exists a para-complex number $w \in \C$ with $z^{1/2} =w$
 if and only if
\begin{equation}\label{eq:root}
  x+y \geq 0 \quad {\rm and} \quad x-y \geq 0. 
\end{equation} 
 In particular ${\ip} ^{1/2}$ does not exist.
 Moreover, for a para-complex number $z = x+ y \ip \in \C$,
 there exists a para-complex number $w \in \C$
 such that $z = e^w$ if and only if
\begin{equation}\label{eq:exp}
  x+y > 0 \quad {\rm and} \quad x-y > 0. 
\end{equation} 
 
 Let $M$ be an orientable connected 2-manifold,
 $G$ a Lorentzian manifold
 and $f:M \to G$ a timelike immersion, that is, the induced metric on $M$ is Lorentzian.
 The induced Lorentzian metric defines a Lorentz conformal structure on $M$:
 for a timelike surface
 there exists a local para-complex coordinate system $z= x+y\ip$
 such that the induced metric $I$ is given by
 $I = e^u dz d\bar z = e^u \left((dx)^2-(dy)^2\right)$.
 Then we can regard $M$ and $f$
 as a Lorentz surface and a conformal immersion, respectively.
 The coordinate system $z$ is called the \textit{conformal coordinate system}
 and the function $e^u$ the conformal factor of the metric with respect to $z$.
 For a para-complex coordinate system $z= x+y\ip$, the partial differentiations
 are defined by
 \[
 \partial_z = \frac12(\partial_x + \ip \partial_y) \quad {\rm and} \quad
 \partial_{\bar z} = \frac12(\partial_x - \ip \partial_y).
\]

\subsection{Structure equations}
 Let $f:M \to \Nil$ be a conformal immersion from a Lorentz surface $M$ into $\Nil$.
 Let us denote the inverse element of $f$ by $f^{-1}$.
 Then the $1$-form $\alpha = f^{-1}df$ satisfies the Maurer-Cartan equation:
\begin{equation}\label{eq:MC}
 d\alpha + \frac12 [\alpha\wedge\alpha] = 0.
\end{equation}
 For a conformal coordinate $z = x + y\ip$
 defined on a simply connected domain $\D \subset M$,
 set $\Phi$ as
\[ \Phi = f^{-1}f_z. \]
 The function $\Phi$ takes values
 in the para-complexification $\mathfrak{nil}_3^{\C}$ of $\mathfrak{nil}_3$.
 Then $\alpha$ is expressed as
\[ \alpha = \Phi dz + \overline{\Phi} d\bar z \]
 and the Maurer-Cartan equation \eqref{eq:MC} as
\begin{equation}\label{eq:MC2}
 \Phi_{\bar z} - \overline{\Phi}_z + [\overline{\Phi}, \Phi] =0.
\end{equation} 
 Denote the para-complex extension of 
 $ds_-^2 =g = \sum_{i, j}g_{ij} dx_i dx_j $ to $\mathfrak{nil}_3^{\C}$
 by the same letter.
 Then the conformality of $f$ is equivalent to
\[
 g(\Phi, \Phi) =0, \quad
 g(\Phi, \overline{\Phi}) >0.
\]
 For the orthonormal basis $\{ e_1, e_2, e_3\}$ of $\mathfrak{nil}_3$
 we can expand $\Phi$ as $\Phi = \phi_1 e_1 + \phi_2 e_2 + \phi_3 e_3$.
 Then the conformality of $f$ can be represented as
\begin{equation}\label{eq:condition}
  -(\phi_1)^2 +(\phi_2)^2 +(\phi_3)^2 =0, \quad
  -\phi_1 \overline{\phi_1} +\phi_2 \overline{\phi_2} +\phi_3 \overline{\phi_3}  =\frac12 e^u,
\end{equation}
 for some function $u$.
 The conformal factor is given by $e^u$.
 Conversely, for a $\mathfrak{nil}_3^{\C}$-valued function $\Phi = \sum_{k=1}^3 \phi_k e_k$
 on a simply connected domain $\D \subset M$
 satisfying \eqref{eq:MC2} and \eqref{eq:condition},
 there exists an unique conformal immersion $f:\D \to \Nil$
 with the conformal factor $e^u$
 satisfying $f^{-1} df = \Phi dz + \overline{\Phi} d\bar z$
 for any initial condition in $\Nil$ given at some base point in $\D$.

 Next we consider the equation for a timelike surface $f$ with constant mean curvature 0.
 For $f$ denote the unit normal vector field by $N$
 and the mean curvature by $H$.
 The tension field $\tau (f)$ for $f$ is given by
 $\tau (f)$ = tr$(\nabla df)$
 where $\nabla df$ is the second fundamental form for $(f,N)$.
 As well known the tension field of $f$ is related to the mean curvature and the unit normal by
\begin{equation}\label{eq:tension}
 \tau(f) = 2HN.
\end{equation}
 By left translating to $(0,0,0)$, we can see this equation rephrased as
\begin{equation}\label{eq:streq}
 \Phi _{\bar z}
 +\overline{\Phi} _z
 + \left\{ \Phi , \overline{\Phi}\right\}
 =
 e^u Hf^{-1} N
\end{equation}
 where $\{\cdot,\cdot\}$ is the bilinear symmetric map defined by 
\[
 \left\{ X , Y\right\}
 =
 \nabla_ {X} Y 
 + \nabla_ {Y} X
\]
 for $X,Y \in \mathfrak{nil}_3$.
 In particular for a surface with the mean curvature 0,
 we have
\begin{equation}\label{eq:minimalcondition}
  \Phi_{\bar z}
 + \overline{\Phi} _z
 + \left\{ \Phi , \overline{\Phi}\right\}=0.
\end{equation}
 
 Conversely, for a $\mathfrak{nil}_3$-valued function $\Phi = \sum_{k=1}^3 \phi_k e_k$
 satisfying \eqref{eq:MC2}, \eqref{eq:condition} and \eqref{eq:minimalcondition}
 on a simply connected domain $\D$,
 there exists a conformal timelike surface $f:\D \to \Nil$
 with the mean curvature $0$
 and the conformal factor $e^u$
 satisfying $f^{-1} df = \Phi dz + \overline{\Phi} d\bar z $
 for any initial condition in $\Nil$ given at some base point in $\D$.
\subsection{Nonlinear Dirac equation for timelike surfaces}\label{sec:deracequation}
 Let us consider the conformality condition of an immersion $f$.
 We first prove the following lemma:
\begin{Lemma}\label{lem:squreroot}
 If a product $xy $ of two para-complex numbers $x, y \in \C$ has
 the square root, then there exists $\epsilon \in \{\pm 1, \pm \ip\}$
 such that $\epsilon x$ and $\epsilon y$ have the square roots.
\end{Lemma}
\begin{proof}
 By the assumption, 
\[
 \Re (x y) \pm \Im (x y) \geq 0
\]
 holds, and a simple computation shows that it is equivalent to 
\[
 (\Re (x)  \pm \Im (x))(\Re (y)  \pm \Im (y)) \geq 0.
\]
 Then the claim follows.
\end{proof}
Since the first condition in  \eqref{eq:condition} can be rephrased as 
\begin{equation}
 \phi_3^2 =  (\phi_1 + i \phi_2)(\phi_1 - i \phi_2),
\end{equation}
 and by Lemma \ref{lem:squreroot}, there exists 
 $\epsilon \in \{\pm 1, \pm \ip \}$ such that 
 $\epsilon(\phi_1 + i \phi_2)$ and $\epsilon(\phi_1 - i \phi_2)$
 have the square roots. Therefore there exist para-complex functions $\overline{\psi_2}$
 and $\psi_1$
 such that 
\[
 \phi_1 + i \phi_2 = 2 \epsilon \overline{\psi_2}^2, \quad
 \phi_1 - i \phi_2 = 2 \epsilon {\psi_1}^2
\]
 hold. Then $\phi_3$ can be rephrased as $\phi_3 = 2 \psi_1 \overline{\psi_2}$.
 Let us compute the second condition in \eqref{eq:condition} by using
 $\{\psi_1, \overline{\psi_2}\}$ as
 \[
 - \phi_1 \overline{\phi_1} +  \phi_2 \overline{\phi_2} + \phi_3 \overline{\phi_3} 
 = -2 \epsilon \bar \epsilon 
 (\psi_1 \overline{\psi_1} - \epsilon \bar \epsilon\psi_2 \overline{\psi_2})^2.
 \]
 Since we have assumed that the left hand side is positive, $\epsilon$ 
 takes values in 
 \[
  \epsilon \in \{\pm \ip\}.
 \]
 Therefore without loss of generality, we have
 \begin{equation}\label{eq:A}
 \phi_1 = \epsilon \left((\overline{\psi_2})^2 +(\psi_1)^2\right), \:
 \phi_2 = \epsilon \ip \left((\overline{\psi_2})^2 -(\psi_1)^2\right), \:
 \phi_3 =  2 \psi_1 \overline{\psi_2}.
\end{equation}
 Then the normal Gauss map $f^{-1}N$ can be represented in terms of the functions $\psi_1$ and $\psi_2$:
\begin{equation}\label{eq:normalgaussmap}
 f^{-1} N = 2 e^{-u/2}
 \left(
 - \epsilon \left( \psi_1 \psi_2 - \overline {\psi_1} \overline {\psi_2} \right) e_1 
 + \epsilon \ip \left( \psi_1 \psi_2 + \overline {\psi_1} \overline {\psi_2} \right) e_2
 -\left( \psi_2 \overline {\psi_2} - \psi_1 \overline {\psi_1} \right) e_3
 \right),
\end{equation}
 where $e^{u/2} = 2(\psi_2 \overline {\psi_2} + \psi_1 \overline {\psi_1})$.
 We can see that,
 using the functions $(\psi_1, \psi_2)$, 
 the structure equations \eqref{eq:MC2} and \eqref{eq:streq}
 are equivalent to the following nonlinear Dirac equation:
\begin{equation}\label{eq:Diracequation}
\left(
 \begin{array}{ll}
 \partial_z \psi_2 + \mathcal U \psi_1\\
 -\partial_{\overline z} \psi_1 + \mathcal V \psi_2
 \end{array}
 \right)
 =
 \left(
 \begin{array}{ll}
 0\\0
 \end{array}
\right).
\end{equation}
 Here the Dirac potential $\mathcal U$ and $\mathcal V$ are given by
\begin{equation}\label{eq:potential1}
 \mathcal U = \mathcal V =
 -\frac{H}2 e^{u/2}+ \frac{\ip}4 h
\end{equation}
  where 
\[
 e^{u/2}=
 2\left(\psi_2 \overline{\psi_2} + \psi_1 \overline{\psi_1} \right) 
 \quad {\rm and} \quad
 h =2\left( \psi_2 \overline{\psi_2}-\psi_1 \overline{\psi_1}  \right).
\]
\begin{Remark}
\mbox{}
\begin{enumerate}
 \item 
 Without loss of generality,
 we can take  $\psi_2 \overline{\psi_2} + \psi_1 \overline{\psi_1}$ as positive value,
 if necessary, by replacing $( \psi_1, \psi_2 )$ into $( -\ip \psi_1, \ip \psi_2 )$.

\item To prove the equations \eqref{eq:MC2} and \eqref{eq:streq}
 from the nonlinear Dirac equation \eqref{eq:Diracequation} with \eqref{eq:potential1},
 the functions $e^{u/2}$ and $h$ in \eqref{eq:potential1} and solutions $\psi_k \:(k=1,2)$
 have to satisfy the relations
\[
 e^{u/2} = 2 (\psi_2 \overline{\psi_2} + \psi_1 \overline{\psi_1}),
 \quad
 h = 2(\psi_2 \overline{\psi_2} - \psi_1 \overline{\psi_1}).
\]  
\end{enumerate}
\end{Remark}
 For a timelike surface with the constant mean curvature $H=0$,
 the Dirac potential takes purely imaginary values.
 Then, by using \eqref{eq:normalgaussmap},
 we have the following lemma. 
\begin{Lemma}
 Let $f:\D \to (\Nil, ds_-^2)$ be a timelike surface with 
 constant mean curvature $H=0$.
 Then the following statements are equivalent:
\begin{enumerate}
 \item The Dirac potential $\mathcal U$ is not invertible at $p \in \D$.
 \item The function $h$ is equal to zero at $p \in \D$.
 \item $E_3$ is tangent to $f$ at $p \in \D$.
\end{enumerate}
\end{Lemma}
\begin{Remark}
 The equivalence between $(2)$ and $(3)$ holds regardless of the value of $H$.
 In general,
 $\mathcal U$ is invertible if and only if
 $(\Re \mathcal U)^2 - (\Im \mathcal U)^2 \neq 0$.
\end{Remark}
 Hereafter we will exclude the points where $\mathcal U$ is not invertible,
 that is, we will restrict ourselves to the case of 
\begin{equation}\label{eq:assumption1}
 \left( {\rm Re}\, \mathcal U \right)^2 - \left( {\rm Im}\, \mathcal U \right)^2 \neq 0.
\end{equation}
 Then, by using \eqref{eq:exp},
 the Dirac potentials can be written as
\begin{equation}\label{eq:potential3}
 \mathcal U =\mathcal V = \tilde \epsilon e^{w/2}
\end{equation}
 for some $\C$-valued function $w$ and $\tilde \epsilon \in \{\pm 1, \pm \ip \}$.
 In particular,
 if the mean curvature is zero
 and the function $h$ has positive values,
 then $\tilde \epsilon = \ip$.

\subsection{Hopf differential and an associated quadratic differential}
 The Hopf differential $A dz^2$
 is the $(2,0)$-part of the second fundamental form
 for $f$, that is,
\[A = g (\nabla _{\partial_z} f_z, N). \]
 A straightforward computation shows that
 the coefficient function $A$ is rephrased in terms of $\psi_k$ as follows:
\[
 A = 2 \{ \psi_1 (\overline{\psi_2})_z - \overline{\psi_2}(\psi_1)_z \}
 - 4 \ip \psi_1^2 (\overline{\psi_2})^2.
\]
 Next we define a para-complex valued function $B$ by
\begin{equation}\label{eq:ARdiff}
 B = \frac14 (2 H - \ip) \tilde A, \quad \mbox{where}\quad  \tilde A = A - 
 \frac{\phi_3^2}{2 H - \ip}.
\end{equation}
 Here $A$ and $\phi_3$ are the Hopf differential and the $e_3$-component of $f^{-1}f_z$ for $f$.
 It is easy to check
 the quadratic differential $B dz^2$ is defined entirely 
 and it will be called the \textit{Abresch-Rosenberg differential}.

\subsection{Lax pair for timelike surfaces}
 The nonlinear Dirac equation can be represented
 in terms of the Lax pair type system.
\begin{Theorem}
 Let $\D$ be a simply connected domain in $\C$ and $f: \D 
 \to \Nil$ a conformal timelike immersion
 for which the Dirac potential $\mathcal U$ satisfies 
 \eqref{eq:assumption1}.
 Then the vector $\widetilde \psi 
 =( \psi_1, \psi_2)$ satisfies the system of equations
\begin{equation}\label{eq:Laxpair}
\widetilde \psi_{z} = \widetilde \psi \widetilde U, \quad 
\widetilde \psi_{\bar z} = \widetilde \psi \widetilde V,
\end{equation}
 where
\begin{align}
 \label{eq:tildeU1}
\widetilde U & =  
\begin{pmatrix}
 \frac12 w_z + \frac12 H_z \tilde \epsilon e^{-w/2} e^{u/2} &
 - \tilde \epsilon e^{w/2}\\
 B \tilde \epsilon e^{-w/2}& 0 
\end{pmatrix},
\\
 \label{eq:tildeV1}
\widetilde V & = 
\begin{pmatrix}
 0 &  - \overline B \tilde \epsilon e^{-w/2}\\ 
 \tilde \epsilon e^{w/2}& \frac12 w_{\bar z} + \frac 1 2 H_{\bar z} \tilde \epsilon e^{-w/2} e^{u/2}
\end{pmatrix}.
\end{align}
 Here, $\tilde \epsilon \in \{\pm1, \pm \ip\}$ is the number decided by \eqref{eq:potential3}.
 Conversely, every solution $\widetilde{\psi}$ to \eqref{eq:Laxpair}
 with \eqref{eq:potential3} and \eqref{eq:potential1}
 is a solution of the nonlinear Dirac equation \eqref{eq:Diracequation} with \eqref{eq:potential1}.
\end{Theorem}
\begin{proof}
 By computing the derivative of the Dirac 
 potential $\tilde \epsilon e^{w/2}$ with respect to $z$,
 we have
\[
 \frac12 w_z \tilde \epsilon e^{w/2}
 =
 - \frac12 H_z e^{u/2}
 + 2 \ip H \psi_1\psi_2 (\overline \psi_2)^2
 -\frac{2H -\ip}2 \psi_2 (\overline \psi_2)_z
 -\frac{2H +\ip}2 \overline{\psi_1} (\psi_1)_z.
\]
 Multiplying the equation above by $\psi_1$
 and using the function $B$ defined in \eqref{eq:ARdiff},
 we derive
\[
 (\psi_1)_z
 =
 \left( \frac12 w_z + \frac12 H_z \tilde \epsilon e^{-w/2}e^{u/2} \right) \psi_1
 +
 B \tilde \epsilon e^{-w/2} \psi_2.
\]
 The derivative of $\psi_2$ with respect to $z$ is given by the nonlinear Dirac equation.
 Thus we obtain the first equation of \eqref{eq:Laxpair}.
 We can derive the second equation of \eqref{eq:Laxpair} in a similar way by differentiating the potential with respect to $\bar z$.
 
 Conversely, if the vector $\widetilde \psi = ( \psi_1, \psi_2)$
 is a solution of \eqref{eq:Laxpair},
 the terms of $(\psi_1)_{\bar z}$ and $(\psi_2)_{z}$ of \eqref{eq:Laxpair}
 are the equations just we want.
\end{proof}
 The compatibility condition of the above system is 
\begin{gather}
 \frac12 w_{z \bar z}
 + e^{w}
 -B \overline{B} e^{-w}
 + \frac12 (H_{z \bar z} + p) \tilde \epsilon e^{-w/2} e^{u/2} =0,\label{GaussEquation} \\
 \overline B_{z} \tilde \epsilon e^{-w/2}
 =
 - \frac12 \overline B H_{z} e^{-w} e^{u/2}
 - \frac12 H_{\bar z} e^{u/2}, \label{Codazzi1}\\
 B_{\bar z} \tilde \epsilon e^{-w/2} = -\frac12 B H_{\bar z} e^{-w} e^{u/2} - \frac12 H_{z} e^{u/2}, \label{Codazzi2}
\end{gather}
 where $p = H_{z}(-w/2 + u/2)_{\bar z}$ for the (1,1)-entry and $p = H_{\bar z}(-w/2 + u/2)_{z}$ for the (2,2)-entry.
 From the above compatibility conditions we have the following:
\begin{Theorem}\label{thm:constant}
 For a constant mean curvature timelike surface in $\Nil$
 which has the Dirac potential invertible anywhere, 
 the Abresch-Rosenberg differential is para-holomorphic.
\end{Theorem}

\begin{Remark}
 To obtain a timelike immersion for solutions $w, B$ and $H$ of the compatibility condition
  \eqref{GaussEquation}, \eqref{Codazzi1} and \eqref{Codazzi2},
 a solution $\widetilde \psi =( \psi_1, \psi_2)$ of \eqref{eq:Laxpair} has to satisfy
\[
 \tilde \epsilon e^{w/2} = -H(\psi_2 \overline{\psi_2} + \psi_1 \overline{\psi_1})
 + \frac{\ip}2 (\psi_2 \overline{\psi_2} - \psi_1 \overline{\psi_1}).
\]
 This gives an overdetermined system and it seems not easy to find 
 a general solution for arbitrary $H$, but for minimal surfaces we will show that 
 it will be automatically satisfied.
\end{Remark}

\section{Timelike minimal surfaces in $\Nil$}
 A timelike surface in $\Nil$ with the constant mean curvature $H=0$
 is called a timelike minimal surface.
 By Theorem \ref{thm:constant},
 the Abresch-Rosenberg differential for a timelike minimal surface
 is para-holomorphic.
 For example
 the triple $B=0, H=0$ and $e^w=16 / (1+16z \bar z)^2$
 is a solution of the compatibility condition
 \eqref{GaussEquation}, \eqref{Codazzi1} and \eqref{Codazzi2}.
 In fact these are derived from a horizontal plane
\begin{equation}\label{eq:horizontalplane}
 f(z) = \left( \frac{2\ip(z- \bar z)} {1+z \bar z}, \frac{2(z + \bar z)} {1 + z \bar z}, 0 \right).
\end{equation}
 Thus the horizontal plane \eqref{eq:horizontalplane} is a timelike minimal surface in $\Nil$.
 We will give examples of timelike minimal surfaces in Section \ref{sc:Example}.
 In this section we characterize timelike minimal surfaces
 in terms of the normal Gauss map.

\subsection{The normal Gauss map} \label{gaussmap}
 For a timelike surface in $\Nil$, the normal Gauss map is given by \eqref{eq:normalgaussmap}.
 Clearly it takes values in de Sitter two sphere $\widetilde{\mathbb{S}}^2_1 \subset \mathfrak{nil}_3$:
\[
 \widetilde{\mathbb{S}}^2_1 = \left\{
 x_1 e_1+ x_2 e_2 + x_3 e_3 \in \mathfrak{nil}_3 \mid
 -x_1^2 + x_2^2 +x_3^2 =1 
 \right\}.
\]
 From now on we will assume that the function $h$ takes positive values,
 that is,
 the image of the normal Gauss map is in lower half part of the de Sitter two sphere.
 Moreover,
 we assume that the timelike surface has the pair of functions $(\psi_1, \psi_2)$
 of the formula \eqref{eq:A} with $\epsilon = \ip$.
 If the function $h$ takes negative values,
 or if the functions $(\psi_1, \psi_2)$ are given
 with $\epsilon = -\ip$,
 by a similar to the case of $h>0$ and $\epsilon = \ip$,
 we can get same results.

 The normal Gauss map $f^{-1}N$ can be considered
 as a map into another de Sitter two sphere
 in the Minkowski space 
\[
 \mathbb{S}^2_1 = \left\{
 (x_1, x_2 , x_3 ) \in \Min \mid
 x_1^2 - x_2^2 +x_3^2 =1 
 \right\}\subset \Min_{(+, -, +)}
\]
 through the stereographic projections
 from $(0, 0, 1) \in \widetilde{\mathbb{S}}^2_1 \subset \mathfrak{nil}_3$:
\[
\pi_{\mathfrak{nil}}^+ :  \mathfrak{nil}_3 \supset \widetilde{\mathbb{S}}^2_1 \ni x_1 e_1 + x_2 e_2 +x_3 e_3
 \mapsto
 \left( \frac{x_1}{1-x_3}, \frac{x_2}{1-x_3}, 0\right) = \frac{x_1}{1-x_3} + \ip \frac{x_2}{1-x_3} \in \C
\]
 and from $(0, 0, -1) \in \mathbb{S}^2_1 \subset \Min_{(+, -, +)}$:
\[
 \pi_{\Min}^- : \Min_{(+, -, +)} \supset \mathbb{S}^2_1 \ni (x_1, x_2, x_3)
 \mapsto
 \left( \frac{x_1}{1+x_3}, \frac{x_2}{1+x_3}, 0\right) = \frac{x_1}{1+x_3} + \ip \frac{x_2}{1+x_3} \in \C.
\]
 In particular, the inverse map $(\pi _ {\Min}^-) ^{-1}$ is given by
\[
(\pi _ {\Min}^-) ^{-1}(g) = \left( \frac{2\Re g}{1+ g \overline{g}}, \frac{2\Im g}{1+ g \overline{g}}, \frac{1- g \overline{g}}{1+ g \overline{g}} \right)
\]
 for $g = \left( \Re g, \Im g, 0 \right) \in \C$.
 Since the normal Gauss map takes values
 in the lower half of the de Sitter two sphere in $\mathfrak{nil}_3$,
 the image under the projection $\pi_{\mathfrak{nil}}^+$ is
 in the region enclosed by four hyperbolas, see Figure \ref{fig:desitter}.
 Two of the four hyperbolas correspond to the vertical points,
 that is, the points where $h$ vanishes,
 and the others correspond to the infinite-points,
 that is, the points where the first fundamental form degenerates.
 Since first and second sign of metrics of $\Nil$ and $\Min_{(+, -, +)}$ are interchanged,
 the image of each hyperbola under the inverse map $(\pi^{-}_{\Min})^{-1}$
 plays the other role.

\begin{figure}[t]
\centering
\includegraphics[width=0.25\linewidth]{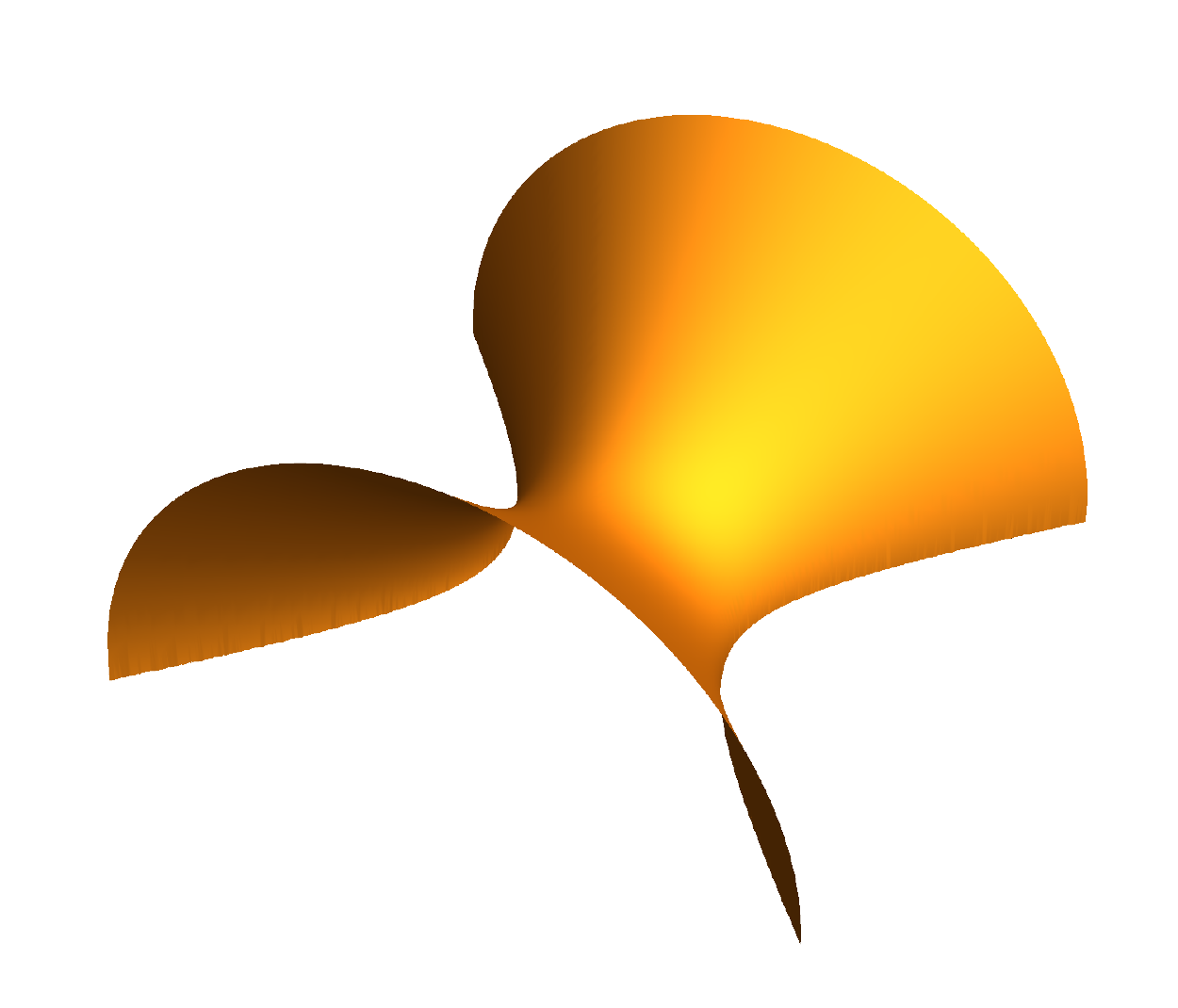}\hspace{0.5cm}
\includegraphics[width=0.3\linewidth]{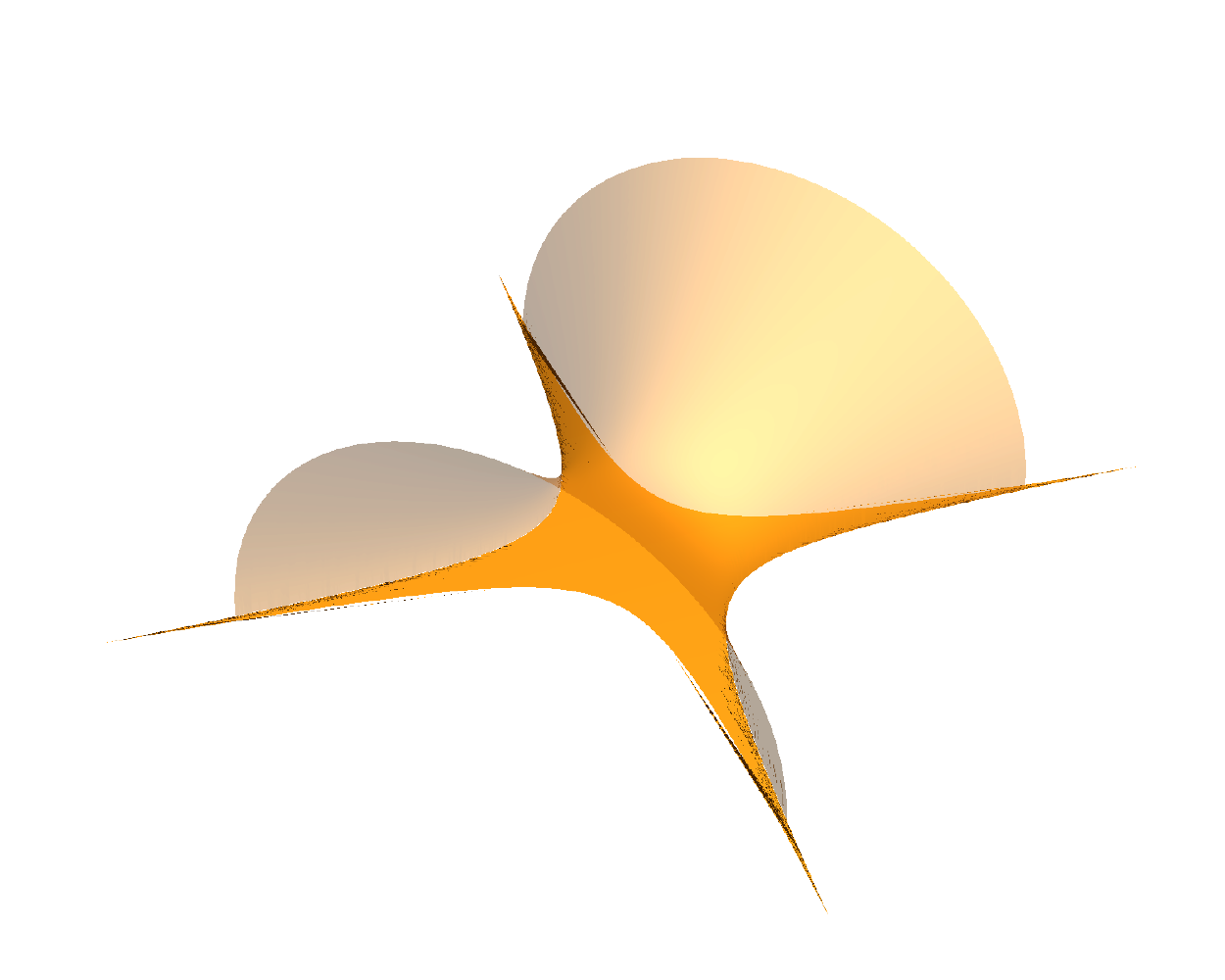}\hspace{0.5cm}
\includegraphics[width=0.3\linewidth]{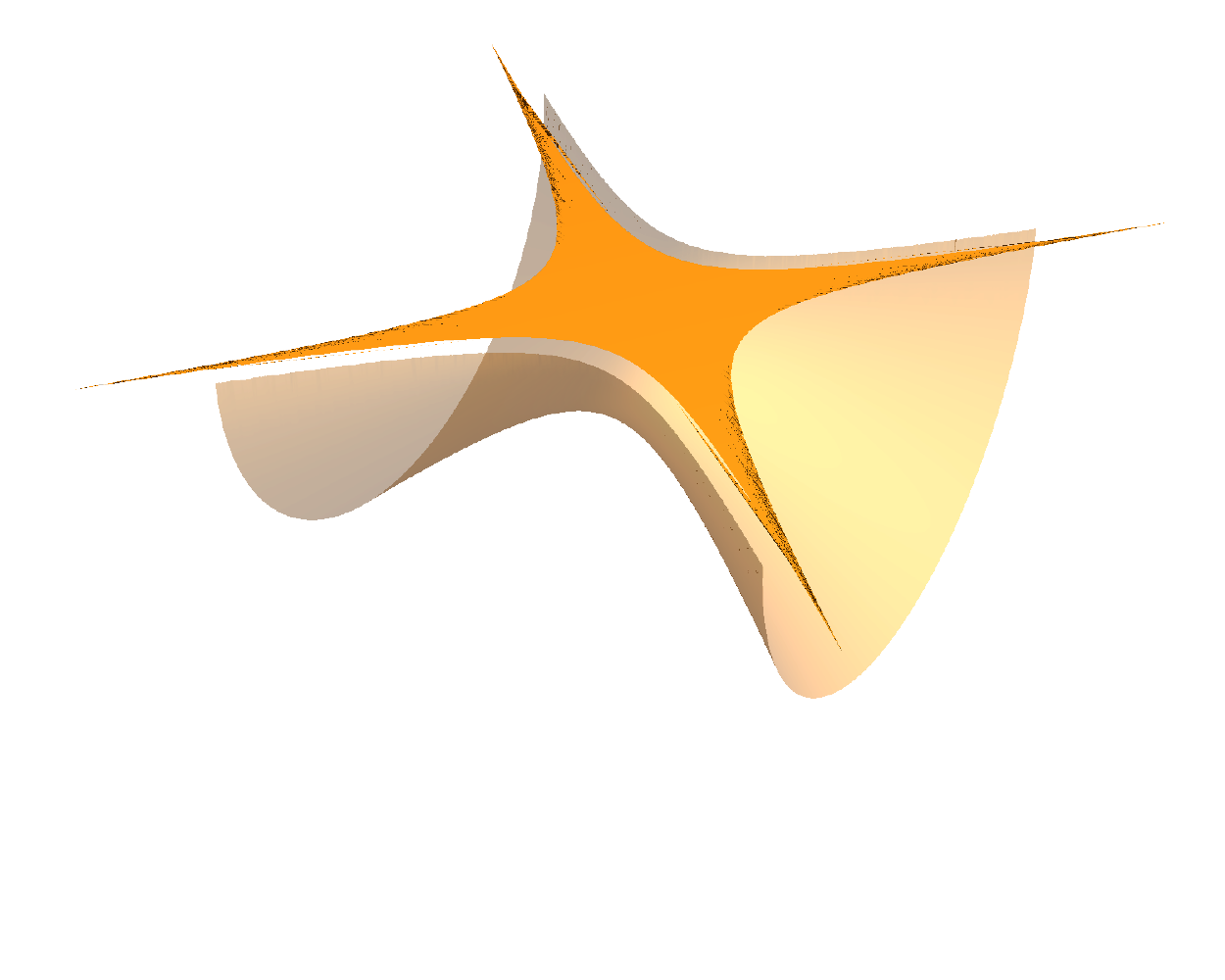}
\caption{The upper half part of the de Sitter two sphere $\mathbb S^2_1$ (left) 
 and its stereographic projection (middle), and 
 the stereographic projection of the lower half part of $\tilde {\mathbb S^{2}_1}$ 
(right).}\label{fig:desitter}
\end{figure}
 
 Define a map $g$ by the composition of the stereographic projection $\pi_{\mathfrak{nil}}^+$ with $f^{-1}N$,
 and then we obtain 
\[
 g = \ip \frac{\overline{\psi_1}}{\psi_2} \in \C.
\]
 Thus the normal Gauss map can be represented as
\[
 f^{-1} N = \frac1{1-g \overline{g}}\left( 2 \Re (g) e_1 + 2 \Im (g)e_2 - (1+g \overline{g})e_3\right)
\]
 and
\begin{equation}\label{eq:A++}
 (\pi _ {\Min}^-) ^{-1} \circ \pi_{\mathfrak{nil}}^+ \circ f^{-1}N
 =
 \frac{1} {\psi_2 \overline{\psi_2} - \psi_1 \overline{\psi_1}}
 \left( -2 \Im(\psi_1 \psi_2), 2 \Re(\psi_1 \psi_2), \psi_2 \overline{\psi_2} + \psi_1 \overline{\psi_1} \right).
\end{equation}
 Let $\isu$ be \textit{the special para-unitary Lie algebra} defined by
\[
 \isu = \left\{
 \begin{pmatrix}
 a\ip & \bar b\\ 
 b & -a\ip
 \end{pmatrix}
 \mid a \in \R, b \in \C
 \right\}
\]
 with the usual commutator of the matrices.
 We assign the following indefinite product on $\isu$:
\[
 \langle X,Y \rangle := 2 {\rm tr} (XY).
\]
 Then we can identify the Lie algebra $\isu$ with $\Min_{(+, -, +)}$ isometrically by
\begin{equation} \label{eq:identification1}
\isu \ni \frac12 
\begin{pmatrix}
 r \ip& -p - q \ip \\
 -p + q \ip & - r \ip
\end{pmatrix}
\longleftrightarrow
(p, q, r) \in \Min_{(+, -, +)}.
\end{equation}
 Let $\ISU$ be the \textit{special para-unitary group} of degree two 
 corresponding to $\isu$:
\[
 \ISU= \left\{
\begin{pmatrix}
 \alpha & \beta \\
 \bar \beta &\bar \alpha
\end{pmatrix}
 \mid
 \alpha,  \beta\in \C, 
 \alpha \bar \alpha -\beta \bar \beta =1
\right\}.
\]
By the identification \eqref{eq:identification1}, the represented normal Gauss map \eqref{eq:A++} is equal to
\[
 (\pi _ {\Min}^-) ^{-1} \circ \pi_{\mathfrak{nil}}^+ \circ f^{-1}N=  \frac{\ip}{2} \ad(F) 
 \begin{pmatrix}
 1&0\\
 0&-1
 \end{pmatrix},
\]
 where $F$ is a $\ISU$-valued map defined by 
\begin{equation}\label{eq:framingF}
F=
\frac{1}{\sqrt{\psi_2 \overline{\psi_2} - \psi_1 \overline{\psi_1}}}
\begin{pmatrix}
 \overline{\psi_2} & \overline{\psi_1}\\
 \psi_1 & \psi_2
\end{pmatrix}.
\end{equation}
The $\ISU$-valued function $F$ defined as above
is called a \textit{frame} of the normal Gauss map $f^{-1}N$.
\begin{Remark}
 In general a frame of the normal Gauss map $f^{-1} N$ is not unique, that is,
 for some frame $F$, there is a freedom of $\ISU$-valued initial 
 condition $F_0$ and $\Uone$-valued map $k$ such that $F_0F k$ is an another frame. 
 In this paper we use the particular frame in \eqref{eq:framingF}, since 
 arbitrary choice of initial condition does not correspond to 
 a given timelike surface $f$.
\end{Remark}
\subsection{Characterization of timelike minimal surfaces}
 Let $F$ be the frame defined in \eqref{eq:framingF} of the normal 
 Gauss map $f^{-1} N$.
 By taking the gauge transformation
\[
 F \mapsto F
\begin{pmatrix}
 e^{-w/4}&0\\
 0&e^{-w/4}
\end{pmatrix},
\]
 we can see the system \eqref{eq:Laxpair} is equivalent to the matrix differential equations
\begin{equation}\label{eq:Laxpair2}
 F_{z} = F U, \quad 
 F_{\bar z} = F V,
\end{equation}
 where
\begin{align*}
  U & =  
\begin{pmatrix}
 \frac14 w_{z} + \frac 1 2  H_z \tilde \epsilon e^{-w/2} e^{u/2}& - \tilde \epsilon e^{w/2} \\
 B \tilde \epsilon e^{-w/2}& -\frac14w_{z} 
\end{pmatrix},\\
 V & = 
\begin{pmatrix}
 -\frac14w_{\bar z} &  - \bar B \tilde \epsilon e^{-w/2}\\ 
 \tilde \epsilon e^{w/2}& \frac14 w_{\bar z} + \frac 1 2 H_{\bar z} \tilde \epsilon e^{-w/2} e^{u/2}
\end{pmatrix}.
\end{align*}
 We define a family of Maurer-Cartan forms $\alpha^{\mu}$ 
 parameterized by $\mu \in \left\{ e^{\ip t} \,|\, t \in \R \right\}$ as follows:
\begin{equation}\label{eq:alphamu}
 \alpha^\mu := U^\mu dz+ V^\mu d\bar z , 
\end{equation}
 where
\begin{align}
 U^\mu & =  
\begin{pmatrix}\label{eq:Umu}
 \frac14 w_{z} + \frac 1 2  H_z \tilde \epsilon e^{-w/2} e^{u/2}&
 - \mu^{-1} \tilde \epsilon e^{w/2} \\
 \mu^{-1} B \tilde \epsilon e^{-w/2}& -\frac14w_{z} 
\end{pmatrix},\\
 V^\mu & = 
\begin{pmatrix}\label{eq:Vmu}
 -\frac14w_{\bar z} &
 - \mu \bar B \tilde \epsilon e^{-w/2}\\ 
 \mu \tilde \epsilon e^{w/2}&
 \frac14 w_{\bar z} + \frac 1 2 H_{\bar z} \tilde \epsilon e^{-w/2} e^{u/2}
\end{pmatrix}.
\end{align}
 
\begin{Theorem}\label{thm:main}
 Let $f$ be a conformal timelike immersion from a simply connected domain 
 $\D \subset \C$ into $\Nil$ 
 satisfying \eqref{eq:assumption1}.
 Then the following conditions are mutually 
 equivalent$:$
\begin{enumerate}
 \item $f$ is a timelike minimal surface.
 \item The Dirac potential $\mathcal U  =\tilde \epsilon e^{w/2} = - \frac{H}2 e^{u/2} + \frac{\ip}4h$ takes
 purely imaginary values.
 \item $d + \alpha^{\mu}$ defines a family of flat connections on 
 $\D \times \ISU$.
 \item The normal Gauss map $f^{-1} N$ is a Lorentz harmonic 
 map into de Sitter two sphere $\mathbb{S}^2_1 \subset \Min_{(+, -, +)}$.
\end{enumerate}
\end{Theorem}
\begin{proof}
 The statement (3) holds if and only if 
\begin{equation} \label{eq:flat}
 (U^\mu)_{\bar z} - (V^\mu)_z +[V^\mu, U^\mu] = 0 
\end{equation} 
for all $\mu \in \left\{ e^{\ip t} \,|\, t \in \R \right\}$.
The coefficients of ${\mu}^{-1}, {\mu}^0$ and $\mu$ of \eqref{eq:flat} are as follows: 
 \begin{eqnarray}
 &\mbox{$\mu^{-1}$-part:}\;\; \frac{1}{2} H_{\bar z} e^{u/2} =0, 
 \;\; B_{\bar z}+\frac{1}{2} B H_{\bar z} \tilde \epsilon e^{-w/2} e^{u/2}=0, \label{eq:codazzi1} \\
 & \mbox{$\mu^{0}$-part:}\;\;
 \frac{1}{2} w_{z \bar z} + e^{w} -B \overline{B} e^{-w}
 + \frac{1}{2}(H_{z \bar z}+p ) \tilde \epsilon e^{-w/2} e^{u/2} =0, \label{eq:structure}\\
 &\mbox{$\mu$-part:} \;\;  
 \overline B_z+\frac{1}{2} \overline B H_z \tilde \epsilon e^{-w/2} e^{u/2}=0, 
  \;\;\frac{1}{2} H_z e^{u/2} =0,\label{eq:codazzi2}
 \end{eqnarray}
 where $p$ is $H_z (-w/2 +u/2)_{\bar z}$ for the $(1, 1)$-entry and
 $H_{\bar z} (-w/2 +u/2)_{z}$ for the $(2, 2)$-entry, respectively.
 Since the equation in \eqref{eq:structure} is a structure equation for 
 the immersion $f$, these are always satisfied, which in fact is equivalent
 to \eqref{GaussEquation}.
 
 The equivalence of (1) and (2) is obvious.\\
 We consider $(1) \Rightarrow (3)$.
 Since $f$ is timelike minimal, by Theorem \ref{thm:constant},
 the Abresch-Rosenberg differential $B dz^2$ is para-holomorphic.
 Hence, the equations \eqref{eq:codazzi1}, \eqref{eq:structure} and \eqref{eq:codazzi2} hold.
 Consequently, the statement $(3)$ holds.

 Next we show $(3) \Rightarrow (1)$.
 Assume that $d + \alpha^\lambda$ is flat, that is, \eqref{eq:codazzi1},
 \eqref{eq:structure} and \eqref{eq:codazzi2} are satisfied.
 Then it is easy to see that $H$ is constant.
 Furthermore, since $\alpha^\mu$ is valued in $\isu$, 
 we can derive that the mean curvature $H$ is 0 
 by comparing (2,1)-entry with (1,2)-entry of $\alpha^\mu$.\\
 Finally we consider the equivalence between (3) and (4).
 The condition (3) is \eqref{eq:flat} and it can be rephrased as 
\begin{equation}\label{eq:Lorentzharm}
 d (* \alpha_{1}) + [\alpha_{0} \wedge *\alpha_1]=0,
\end{equation}
 where $\alpha_0 = \alpha^{\prime}_{\mathfrak k} dz 
 + \alpha^{\prime \prime}_{\mathfrak k} d\bar z $ and
 $\alpha_1 = \alpha^{\prime}_{\mathfrak m} dz 
 + \alpha^{\prime \prime}_{\mathfrak m} d\bar z $
 and $\isu$ has been decomposed as $\isu = \mathfrak k + \mathfrak m$
 with 
\[
 \mathfrak k = \left\{ \begin{pmatrix} \ip r & 0\\ 0 & -\ip r\end{pmatrix}\mid r \in \R
\right\}, \quad 
 \mathfrak m = \left\{ \begin{pmatrix} 0& - p - q \ip \\ - p + q \ip  &0
 \end{pmatrix}\mid p, q \in \R
\right\}.
\] 
 Moreover $*$ denotes the Hodge star operator defined by 
\[
 * d z = \ip dz, \quad 
 * d \bar z = -\ip d\bar z.
\]
 It is known that by \cite[Section 2.1]{Melko-Sterling}, the harmonicity condition
 \eqref{eq:Lorentzharm} is 
 equivalent to the Lorentz harmonicity of the normal Gauss map $f^{-1} N = 
 \tfrac{\ip}{2} F \sigma_3 F^{-1}$
 into the symmetric space $\mathbb S^2_1$.
 Thus the equivalence between (3) and (4) follows.
\end{proof}
 From Theorem \ref{thm:main}, we define the followings$:$ 
\begin{Definition} \label{def:extendedframe}
\mbox{}
\begin{enumerate}
 \item 
 For a timelike minimal surface $f$ 
 in $\Nil$ with the frame $F$ in \eqref{eq:framingF}
 of the normal Gauss map,
 let $F^{\mu}$ be a $\ISU$-valued solution 
 of the matrix differential equation 
 $( F^{\mu} )^{-1} d F^{\mu} = \alpha^{\mu}$
 with $F^{\mu} |_{\mu=1} = F$.
 Then $F^{\mu}$ is called an \textit{extended frame} of the timelike minimal surface 
 $f$.
\item Let $\tilde F^{\mu}$ be a $\ISU$-valued solution of 
 $( \tilde F^{\mu} )^{-1} d\tilde F^{\mu} = \alpha^{\mu}$. 
 Then $\tilde F^{\mu}$ is called a \textit{general extended frame}.
\end{enumerate}
\end{Definition}
 Note that an extended frame $F^{\mu}$ and a general extended frame $\tilde F^{\mu}$
 are differ by an initial condition $F_0$, $\tilde F^{\mu}=F_0F^{\mu}$, and 
 $F^{\mu}$ can be explicitly written as
\begin{equation}\label{eq:extendedframe}
 F^\mu = \frac1{\sqrt{\psi_2(\mu) \overline{\psi_2(\mu)} - \psi_1(\mu) \overline{\psi_1(\mu)}}}
\begin{pmatrix}
 \overline{\psi_2(\mu)} & \overline{\psi_1(\mu)}\\
 \psi_1(\mu) & \psi_2(\mu)
\end{pmatrix}, \quad 
\end{equation}
 where $\psi_j(\mu =1)= \psi_j \;(j=1, 2)$ are the original generating spinors of
 a timelike minimal surface $f$.
 For a timelike minimal surface, the Maurer-Cartan form 
 $\alpha^{\mu} = U^{\mu} dz + V^{\mu} d \bar z$ of a general extended frame 
 $\tilde F^{\mu}$ can be written 
 explicitly as follows:
\begin{equation}\label{eq:UVmu}
 U^{\mu} = \begin{pmatrix} \frac12 (\log h)_z& -\frac{\ip}{4} h  \mu^{-1}\\4 {\ip } B h^{-1} \mu^{-1}& -\frac12 (\log h)_z\end{pmatrix}, 
 \quad 
 V^{\mu} = \begin{pmatrix} -\frac12 (\log h)_{\bar z} & - 4 {\ip }\bar B h^{-1}\mu \\ \frac{\ip}{4} h\mu&\frac12 (\log h)_{\bar z}\end{pmatrix}.
\end{equation}

\section{Sym formula and duality between timelike minimal surfaces in three-dimensional Heisenberg group  and timelike CMC surfaces in Minkowski space}\label{sc:Sym}
 In this section we will derive an immersion formula for timelike minimal surfaces in $\Nil$ in terms of the extended frame, the so-called \textit{Sym formula}. Unlike 
 the integral represetnation formula, the so-called \textit{Weierstrass type representation}
 \cite{LMM, SSP, CO, Kondreak}, the Sym-formula will be given by the derivative of the extended frame with respect 
 to the spectral parameter. 

 We define a map $\Xi : \isu \to \mathfrak{nil}_3$ by
\begin{equation}\label{eq:linearisom1}
 \Xi \left(
 x_1 \mathcal{E}_1 + x_2 \mathcal{E}_2 + x_3 \mathcal{E}_3
 \right)
 :=
 x_1e_1 + x_2e_2 + x_3e_3
\end{equation}
where
\begin{equation}\label{eq:E}
 \mathcal{E}_1 = \frac12
 \begin{pmatrix}
 0&-\ip \\
 \ip &0
 \end{pmatrix},
 \quad
 \mathcal{E}_2 =\frac12
 \begin{pmatrix}
 0&-1 \\
 -1 &0
 \end{pmatrix},
 \quad
 \mathcal{E}_3 = \frac12
 \begin{pmatrix}
 \ip&0\\
 0&-\ip
 \end{pmatrix}.
\end{equation}
 Clearly, $\Xi$ is a linear isomorphism but not a Lie algebra isomorphism.
 Moreover, define a map $\Xi_{\mathfrak{nil}} : \isu \to \Nil$ 
 as $\Xi_{\mathfrak{nil}} = \exp \circ \Xi$,
 explicitly 
\begin{equation}
 \Xi_{\mathfrak{nil}} 
 \left(
 \frac12 
 \begin{pmatrix}
 x_3 \ip & -x_2 - x_1 \ip\\ 
 -x_2 + x_1 \ip & -x_3 \ip
 \end{pmatrix}
 \right)
 =
 (x_1, x_2, x_3).
\end{equation}
 Then we can obtain a family of timelike minimal surfaces in $\Nil$
 from an extended frame of a timelike minimal surface.
\begin{Theorem} \label{thm:Sym1}
 Let $\D$ be a simply connected domain in $\C$
 and $F^\mu$ be an extended frame defined in 
 \eqref{eq:extendedframe} for some conformal timelike minimal surface on $\D$ 
 for which the functions $\psi_1, \psi_2$ are given by the formula \eqref{eq:A} with $\epsilon = \ip$
 and the function $h$ defined by \eqref{eq:potential1} has positive values on $\D$.

 Define maps $f_{\Min}$ and $N_{\Min}$ 
 respectively by
 \begin{equation}\label{eq:SymMin}
 f_{\Min}=-\ip \mu (\partial_{\mu} F^{\mu}) (F^{\mu})^{-1} 
 - \frac{\ip}{2} \ad (F^{\mu}) \sigma_3
 \quad \mbox{and} \quad
 N_{\Min}= \frac{\ip}{2} \ad (F^{\mu}) \sigma_3,
 \end{equation}
 where $\sigma_3 = \left( \begin{smallmatrix} 1 &0 \\ 0 & -1 \end{smallmatrix}\right)$.
 Moreover, define a map $f^{\mu}:\mathbb{D}\to \mathrm{Nil}_3$ by%
\begin{equation}\label{eq:symNil}
 f^{\mu}:=\Xi_{\mathrm{nil}}\circ \hat{f}
 \quad \mbox{with} \quad
 \hat f = 
    (f_{\Min})^o -\frac{\ip}{2} \mu (\partial_{\mu}  f_{\Min})^d, 
\end{equation}
 where the superscripts ``$o$'' and ``$d$'' denote the off-diagonal and 
 diagonal part, 
 respectively. 
 Then, for each $\mu \in \S^1_1 = \{e^{\ip t} \in \C \mid t \in \R \}$ 
 the following statements hold$:$
\begin{enumerate}
 \item The map $f^{\mu}$ is a 
 timelike minimal surface (possibly singular) in $\Nil$ and 
 $N_{\Min}$ is the isometric image of the normal Gauss map of $f^{\mu}$.
 Moreover, $f^{\mu} |_{\mu=1}$ and the original surface are same up to a translation.

 \item The map $f_{\Min}$ is a timelike constant mean curvature 
 surface with mean curvature $H=1/2$ in $\Min$
 and $N_{\Min}$ is the spacelike unit normal vector of $f_{\Min}$.
\end{enumerate}
\end{Theorem}

\begin{proof}
 Because of the continuity of the extended frame with respect to the parameter $\mu$,
 $F^\mu$ can be represented in the form of
\[
 F^\mu = \frac1{\sqrt{\psi_2(\mu) \overline{\psi_2(\mu)} - \psi_1(\mu) \overline{\psi_1(\mu)}}}
\begin{pmatrix}
 \overline{\psi_2(\mu)} & \overline{\psi_1(\mu)}\\
 \psi_1(\mu) & \psi_2(\mu)
\end{pmatrix}
\]
 for some $\C$-valued functions $\psi_1(\mu)$ and $\psi_2(\mu)$ with $\psi_k(1) = \psi_k$ for $k=1, 2$.
 Since $F^\mu$ satisfies the equations
 \[F^\mu_z = F^\mu U^\mu, \quad F^\mu_{\bar z} = F^\mu V^\mu,\]
 with \eqref{eq:Umu}, \eqref{eq:Vmu} and $H=0$,
 by considering the gauge transformation
 \[ F^\mu \mapsto F^\mu
 \begin{pmatrix}
 \mu^{-1/2}&0\\
 0&\mu^{1/2}
 \end{pmatrix},\]
 it can be shown that
 the deformation with respect to parameter $\mu$ does not change the Dirac potential,
 that is,
 $\psi_2(\mu) \overline{\psi_2(\mu)} - \psi_1(\mu) \overline{\psi_1(\mu)}$ is independent of $\mu$.
 
 Since $F^\mu$ is $\ISU$-valued,
 a straightforward computation shows that
 $\ip \mu (\partial _{\mu} F^\mu) (F^\mu)^{-1}$
 and $N_{\Min}$ take values in $\isu$.
 Hence $f_{\Min}$ is a $\isu$-valued map.
 Therefore,
 the diagonal entries of $\ip \mu (\partial_{\mu} f_{\Min})$ take purely imaginary values
 and the trace of $\ip \mu (\partial_{\mu} f_{\Min})$ vanishes.
 Thus $\ip \mu (\partial_{\mu} f_{\Min})^d$ takes $\isu$ values.
 
 Next we compute $\partial_z \hat f$.
 By the usual computations we obtain
\begin{eqnarray}
 \partial_z f_{\Min}
 &=&
 \partial_z 
 \left( -\ip \mu (\partial_{\mu} F^\mu) (F^\mu)^{-1}) 
 - \frac{\ip}{2} \ad (F^{\mu}) \sigma_3\right) \notag\\
 &=&
 \ad(F^\mu) 
 \left( -\ip \mu (\partial_\mu U^\mu) 
 - \frac {\ip}2 
 \left[U^\mu, 
 \sigma_3
 \right]
 \right)\notag\\
 &=&
 -2\mu^{-1}e^{w/2} \ad(F^\mu) \label{eq:zfmin}
 \sigma_3\\
 &=&
 \mu^{-1}
\begin{pmatrix}
 \psi_1(\mu) \overline{\psi_2(\mu)} & - (\overline{\psi_2(\mu)})^2\\
 (\psi_1(\mu))^2 & -\psi_1(\mu) \overline{\psi_2(\mu)}
\end{pmatrix}\notag.
\end{eqnarray}
Then we have
\begin{eqnarray}
 \partial_z f_{\Min} 
 &=& \frac12
\begin{pmatrix}
 \phi_3(\mu) & -\phi_2(\mu) - \ip \phi_1(\mu) \\
 -\phi_2(\mu) + \ip \phi_1(\mu) & -\phi_3(\mu)
\end{pmatrix} \notag\\
&=&\label{eq:f_z}
 \phi_1(\mu) \mathcal{E}_1 + \phi_2(\mu) \mathcal{E}_2 + \ip \phi_3(\mu) \mathcal{E}_3 \label{eq:fz}
\end{eqnarray}
with
\[
 \phi_1(\mu) = \mu^{-1} \ip \left( (\overline{\psi_2(\mu)})^2 + (\psi_1(\mu)^2) \right),
 \quad
 \phi_2(\mu) = \mu^{-1} \left( (\overline{\psi_2(\mu)})^2 - (\psi_1(\mu)^2) \right)
\]
and\[ \phi_3(\mu) = \mu^{-1} 2 \psi_1(\mu) \overline{\psi_2(\mu)}.
\]
By using \eqref{eq:zfmin}, we can compute 
\begin{eqnarray*}
 \partial_z \left( -\frac{\ip}2 \mu \left( \partial_{\mu} f_{\Min} \right)  \right)
 &=&
 -\frac{\ip}2 \mu \partial_{\mu} (\partial_z f_{\Min})\\
 &=&
 \ip e^{w/2} \mu (-\mu^{-2}) \ad(F^\mu) 
\begin{pmatrix}
 0&1\\
 0&0
\end{pmatrix}\\
 && \,+ \ip e^{w/2} \left[ \ip \mu^{-1} (-f_{\Min} - N_{\Min}), -\frac12 \mu e^{-w/2} \partial_z f_{\Min} \right]\\
 &=&
 \frac{\ip}2 \partial_z f_{\Min}
 + \left[f_{\Min} + N_{\Min}, \frac12 \partial_z f_{\Min} \right].
\end{eqnarray*}
Using \eqref{eq:zfmin}, we have
\[
\left[f_{\Min}, \frac12 \partial_z f_{\Min} \right]^d
=\frac12 \left( \phi_2 (\mu) \int \phi_1 (\mu) dz - \phi_1 (\mu) \int \phi_2 (\mu) dz \right) \mathcal E_3
\]
and
\[
\left[ N_{\Min}, \frac12 \partial_z f_{\Min} \right]
= \frac{\ip}2 \partial_z f_{\Min}.
\] 
 Consequently, we have
\[
 \partial_z \left( -\frac{\ip}2 \mu \left( \partial_{\mu} f_{\Min} \right)  \right)^d
 =
 \left(
 \phi_3(\mu) + \frac12 \left( \phi_2 (\mu) \int \phi_1 (\mu) dz - \phi_1 (\mu) \int \phi_2 (\mu) dz\right)
 \right)
 \mathcal E_3.
\]
 Thus we obtain
\begin{eqnarray*}
 \partial_z \hat f 
 &=& 
 \partial_z (f_{\Min})^o
 + \partial_z \left( -\frac{\ip}2 \mu \left( \partial_{\mu} f_{\Min} \right)  \right)^d\\
 &=&
 \phi_1(\mu) \mathcal E_1 + \phi_2(\mu) \mathcal E_2
 +
 \left(
 \phi_3(\mu) + \frac12 \left( \phi_2 (\mu) \int \phi_1 (\mu) dz - \phi_1 (\mu) \int \phi_2 (\mu) dz\right)
 \right)
 \mathcal E_3
\end{eqnarray*}
 and then
\begin{equation}\label{eq:conclusion1}
 (f^{\mu})^{-1} (\partial_z f^{\mu})
 = \phi_1(\mu) e_1 + \phi_2(\mu) e_2 + \phi_3(\mu) e_3.
\end{equation}

 The equation \eqref{eq:conclusion1} means that,
 for $\mu = e^{\ip t}$ with sufficiently small $t \in \R$,
 the map $f^\mu$ is conformal with the conformal parameter $z$
 and the conformal factor $4(\psi_2(\mu) \overline{\psi_2(\mu)} + \psi_1(\mu) \overline{\psi_1(\mu)})^2$. 
 To complete the proof of $(1)$ we check the mean curvature and the normal Gauss map of $f^{\mu}$.
 Since the Dirac potential of $f^\mu$ is same with the one of the original timelike minimal surface,
 the mean curvature of $f^\mu$ is zero
 for $\mu$ with $\psi_2(\mu) \overline{\psi_2(\mu)} + \psi_1(\mu) \overline{\psi_1(\mu)}$
 nowhere vanishing on $\D$.
 Using the map $\mathfrak{nil}_3 \supset \widetilde{\mathbb{S}}^2_1 \to \mathbb{S}^2_1 \subset \Min_{(+,-,+)}$
 defined in Section \ref{gaussmap},
 the normal Gauss map of $f^\mu$ is converted into $N_{\Min}$.
 To prove $(2)$, see Appendix \ref{timelikemin}.
\end{proof}

\begin{Remark}
 In other cases, $h<0$ or $\epsilon = - \ip$,
 we can get the same result with Theorem \ref{thm:Sym1}
 by adapting the identification \eqref{eq:identification1} between $\isu$ and $\Min_{(+,-,+)}$
 and the linear isomorphism \eqref{eq:linearisom1} from $\isu$ to $\mathfrak{nil}_3$
 precisely.
 For example,
 when the original timelike minimal surface has $h>0$ and $\epsilon = - \ip$,
 we should replace the identification \eqref{eq:identification1}
 and the linear isomorphism \eqref{eq:linearisom1}, respectively,
 into
\[
\isu \ni \frac12 
\begin{pmatrix}
  r \ip& - (-p - q \ip) \\
 - (-p + q \ip) & - r \ip
\end{pmatrix}
\longleftrightarrow
(p, q, r) \in \Min_{(+, -, +)}
\]
 and
\[
 \Xi \left(
 x_1 \mathcal{E}_1 + x_2 \mathcal{E}_2 + x_3 \mathcal{E}_3
 \right)
 :=
 -x_1e_1 - x_2e_2 + x_3e_3
\]
where $\mathcal{E}_j\;(j=1, 2, 3)$ is 
 defined in \eqref{eq:E}.
\end{Remark}
  In Theorem \ref{thm:Sym1}, we recover a given timelike minimal surface in 
 $\Nil$ in terms of generating spinors and Sym formula.
 More generally, we can construct timelike minimal surfaces
 using  a non-conformal harmonic map into $\mathbb S^2_1$.
 As we have seen in the proof of Theorem \ref{thm:Sym1}
 the harmonicity of a map $N$ into $\mathbb S^2_1$
 in terms of 
\[
 d(* \alpha_1) + [\alpha_0 \wedge *\alpha_1] =0,
\]
 where $\alpha$ is the Maurer-Cartan form of 
 the frame $\tilde{F} : \D \to \ISU$ of $N$ and 
 moreover, $\alpha = \alpha_0 + \alpha_1$ is the representation
 in accordance with the decomposition $\isu = \mathfrak k + \mathfrak m$.
 Denote the $(1,0)$-part and $(0,1)$-part of $\alpha_1$ by
 $\alpha_{1} {'}$ and $\alpha_{1} {''}$,
 and define a $\isu$-valued $1$-form $\alpha^{\mu}$ for each $\mu \in S^1_1$ by
\[
 \alpha^{\mu} := \alpha_0 + \mu^{-1} \alpha_{1} {'} + \mu \alpha_1 {''}.
\]
 Then $\alpha^{\mu}$ satisfies 
\[
 d \alpha^{\mu} + \frac12 [\alpha^{\mu} \wedge \alpha^{\mu}] =0
\]
 for all $\mu \in S^1_1$,
 and thus there exists $\tilde{F}^{\mu} : \D \to \ISU$
 which is smooth with respect to the parameter $\mu$
 and satisfies $(\tilde{F}^{\mu})^{-1} d \tilde{F}^{\mu} = \alpha^{\mu}$ for each $\mu$.
 Thus $\tilde F^{\mu}$ is the extended frame of the harmonic map $N$.
 As well as Theorem \ref{thm:Sym1} we can show the following theorem:
\begin{Theorem} \label{thm:Sym2}
 Let $\tilde{F}^\mu : \D \to \ISU$ be the extended frame of 
 a harmonic map $N$ into the $\mathbb S^2_1$.
 Assume that the coefficient function $a$ of $(1, 2)$-entry of $\alpha_1 {'}$ satisfies 
 $a \overline a <0$ on $\D$.
 Define the maps $\tilde f_{\Min}$, $\tilde N_{\Min}$ 
 and $\tilde f^{\mu}$ respectively by the Sym formulas in \eqref{eq:SymMin} 
 and \eqref{eq:symNil} where $F^{\mu}$ 
 replaced by $\tilde F^{\mu}$.
 Then, under the identification \eqref{eq:identification1}of $\isu$ and $\Min$
 and the linear isomorphism \eqref{eq:linearisom1} from $\isu$ to $\mathfrak {nil}_3$,
 for each $\mu = e^{\ip t} \in \mathbb{S}^1_1$ the following 
 statements hold$:$
\begin{enumerate}
 \item The map $\tilde f_{\Min}$ is a timelike constant mean curvature 
 surface with mean curvature $H=1/2$ in $\Min$ with the first fundamental form 
 $I=-16 a \overline a dz d\bar z$
 and $\tilde N_{\Min}$ is the spacelike unit normal vector of $\tilde f_{\Min}$.

 \item The map $\tilde f^{\mu}$ is a timelike minimal surface  
 (possibly singular) in $\Nil$
 and 
 $N_{\Min}$ is the isometric image of the normal Gauss map of $f^{\mu}$.
 In particular, $\tilde{F}^\mu$ is an extended frame of
 some timelike minimal surface $f$. 
\end{enumerate}
\end{Theorem}
\begin{proof}
 To prove the theorem, one needs to define generating spinors properly:
 After gauging the extended frame the upper right corner of $\alpha_1^{\prime}$
 takes values in purely imaginary, that is $a$ can be assumed to be purely imaginary.
 Define $h$ by $h = -4 \ip a$, and $\tilde \psi_1$ and $\tilde \psi_2$ by 
putting 
\[
 \tilde F_{21} = \sqrt{2} \tilde \psi_1 h^{-1/2}, \quad 
 \tilde F_{22} = \sqrt{2} \tilde \psi_2 h^{-1/2},
\]
 respectively. Then $\tilde \psi_1$ and $\tilde \psi_2$ are 
 generating spinors of the map $\tilde f^{\mu}$ and its angle function is exactly 
 $h = 2 \left(\tilde \psi_2 \overline{\tilde \psi_2} -\tilde \psi_1 \overline{\tilde \psi_1}\right)$.
\end{proof}

\section{Generalized Weierstrass type representation 
for timelike minimal surfaces in $\Nil$} \label{sc:Generalized Weierstrass type representation}
 In this section we will give a construction of timelike minimal surfaces in $\Nil$
 in terms of the para-holomorphic data, the so-called \textit{generalized Weierstrass 
 type representation.}  The heart of the construction is based on two loop group 
 decompositions, the so-called \textit{Birkhoff} and \textit{Iwasawa} decompositions, 
 which are reformulations of \cite[Theorem 2.5]{DIT}, see also \cite{PreS:LoopGroup},
 in terms of the para-complex structure.

\subsection{From minimal surfaces to normalized potentials: The Birkhoff decomposition}
 Let us recall the hyperbola on $\C$:
\begin{equation}
 \S_1^1=\{ \mu \in \C\mid \mu \bar \mu =1, \, \Re \mu >0 \}.
\end{equation}
  Since an extended frame $F^{\mu}$ is analytic on $\S_1^1$ (in fact it is 
 analytic on $\C \setminus \{x (1\pm \ip)\mid x \in \R\}$), 
 it is natural to introduce the following loop groups$:$
\begin{align*}
\LSLC &= 
 \left\{
 g : \S_1^1 \to \SLC \mid 
\begin{array}{l}
 g =\cdots+ g_{-1} \mu^{-1} + g_0 + g_1 \mu^{1} + \cdots\\
\mbox{and
  $g(- \mu) = \sigma_3 g(\mu)\sigma_3$}
\end{array} 
 \right\}, \\
\LSLCP &= 
 \left\{ g \in  \LSLC
 \mid g = g_0 + g_1 \mu^1 + \cdots
 \right\}.
\end{align*}
 On the one hand, we define 
\[
 \LSLCN= 
 \left\{ g \in  \LSLC 
 \mid g= g_0 + g_{-1} \mu^{-1} + \cdots
 \right\}.
\]
 We now use the lower subscript $*$ for normalization
  at $\mu =0$ or $\mu = \infty$ by identity, that is 
 \[
 \Lambda^{\prime \pm}_* {\rm SL}_{2} \C_{\sigma}= \left\{
 g \in \Lambda^{\prime \pm} {\rm SL}_{2} \C_{\sigma} \mid
\mbox{$g(0)=\id$ for 
 $\LSLCP$ or $g(\infty)=\id$ for 
 $\LSLCN$}
\right\}.
\]
 Moreover, we define the loop group of the special 
para-unitary group $\ISU:$
\[
 \LISU =
\left\{
 g \in   \LSLC  \mid 
 \sigma_3 \left(\overline{g(1/\bar \mu)}^{T}\right)^{-1} \sigma_3 = g(\mu)
 \right\}.
\]
 Further, let us introduce the following subgroup
\[
 \Uone= \left\{ \di( e^{\ip \theta}, e^{- \ip \theta}) 
\mid \theta \in \R \right\}.
\]
 The fundamental decompositions for the above loop groups are
 Birkhoff and Iwasawa decompositions as follows$:$
\begin{Theorem}[Birkhoff and Iwasawa decompositions]\label{thm:BIdecomposition}
 The 
 loop group $\LSLC$ can be decomposed as follows$:$
\begin{enumerate}
 \item[(1)]{\rm Birkhoff decomposition}$:$
 The multiplication maps 
\begin{equation}\label{eq:Birkhoff}
 \LSLCNN \times 
 \LSLCP \to  \LSLC
\quad \mbox{and}\quad 
 \LSLCPN\times 
 \LSLCN\to \LSLC
\end{equation}
 are diffeomorphism onto the open dense subsets of $\LSLC$,
 which will be called the {\rm big cells} of $\LSLC$.
 \item[(2)]{\rm Iwasawa decomposition}$:$
 The multiplication map 
 \begin{equation}\label{eq:Iwasawa}
 \LISU \times   
 \LSLCP\to \LSLC
 \end{equation}
 is an diffeomorphism onto the open dense subset of $\LSLC$,
 which will be called the {\rm big cell} of $\LSLC$.
\end{enumerate}
\end{Theorem}
\begin{proof}
 We first note that a given real Lie algebra $\mathfrak g$, the para-complexification 
 $\mathfrak g \otimes \C$ of $\mathfrak g$ is isomorphic to $\mathfrak g \oplus 
 \mathfrak g$ as a real Lie algebra, that is, the isomorphism is given explicitly 
 as
\begin{equation}\label{eq:doubleiso}
\mathfrak g \oplus 
 \mathfrak g \ni (X, Y) \mapsto \frac12 (X+Y) + \frac12(X-Y)\ip  \in
 \mathfrak g \otimes \C. 
\end{equation}
 Accordingly an isomorphism between $\SLR \times \SLR$ and $\SLC$ follows.
 In particular we have an isomorphism between $\{\di(a, a^{-1})\mid a\in \R^{\times}\}
 \times \{\di(a, a^{-1})\mid a\in \R^{\times}\}$
 and $\{\di(r e^{\ip \theta}, r^{-1} e^{- \ip \theta})\mid r \neq 0, \; \theta \in \R\}$ follows.
 Let us consider two real Lie algebras $\slR$ and $\isu$:
\[
 \slR = \left\{\begin{pmatrix}
 a & b \\ c & -a 
	       \end{pmatrix}\mid a, b, c \in \R\right\}, \quad 
 \isu = \left\{\begin{pmatrix}
 c \ip & b - a \ip \\ b + a \ip & -c \ip 
	       \end{pmatrix}\mid a, b, c \in \R\right\}.
\]
 Then an explicit map 
\begin{equation}\label{eq:isom}
 X \mapsto \frac12 (X + X^*) + \frac{1}{2} (X- X^*)\ip, \quad 
 X^*= - \sigma_3 \overline{X}^T \sigma_3
\end{equation}
 induces an isomorphism between $\slR$ and $\isu$.
 Note that $X^* =- \sigma_3 X\sigma_3$ for $X \in \slR$.
 Then accordingly an isomorphism between $\SLR$ and $\ISU$ follows.

 Let us now define the loop algebras of $\mathfrak {sl}_2 \R$
 by 
\begin{align*}
 \lslR & = \left\{\xi : \R^{+} \to 
 \mathfrak {sl}_{2} \R \mid \mbox{$\xi = \cdots + \xi_{-1}\lambda^{-1} + 
 \xi_0 + \xi_1 \lambda + \cdots $ and $\xi(- \lambda) = \sigma_3 \xi(\lambda) \sigma_3$}
\right\}, \\
 \Lambda^{\pm} \mathfrak {sl}_2 \R_{\sigma} & = \left\{ 
 \xi \in  \lslR \mid \xi=  \xi_0 + \xi_{\pm 1} \lambda^{\pm 1} + \cdots 
\right\}.
\end{align*}
 Moreover, the lower subscript $*$ denotes normalization at $\lambda = 0$ and 
 $\lambda = \infty$, that is, $\xi_0=0$ in $ \Lambda^{\pm} \mathfrak {sl}_2 \R$.
 On the one hand the loop algebra of $\isu$
 is defined by 
\[
  \lisu= \left\{ \tau : \S^1_1 \to \isu \mid 
 \tau(- \mu) = \sigma_3 \tau(\mu) \sigma_3\right\}.
\]
 The Lie algebra of 
 $\LSLC$ is defined by
\begin{align*}
 \lslc = \left\{ \tau : \S^1_1 \to \slc \mid 
 \mbox{$\tau = \cdots + \tau_{-1} \mu^{-1} + \tau_0 + \tau_1 \mu + \cdots 
 $ and $\tau(- \mu) = \sigma_3 \tau(\mu) \sigma_3$}
\right\},
\end{align*}
 and 
 it is easy to see that the loop algebra $\lisu$ can be extended 
 to the following fixed point set of an anti-linear 
 involution of $\lslc$:
\[
 \lisu=\left\{ \tau \in   \lslc \mid 
 \tau^{*}(1/\bar \mu) = \tau (\mu)  \right\}.
\]
 We now identify the two loop algebras $\lslR$ and $\lisu$ as follows:
 Let $\xi
 = \cdots + \xi_{-1} \lambda^{-1} + \xi_0 + \xi_1 \lambda + \cdots$ 
 with $\xi_{i} \in \slR$ be an element in $\lslR$ and consider the isomorphism 
 \eqref{eq:isom}: 
 \begin{align*}
  \xi \mapsto 
 \tilde \xi&= \xi \ell+\xi^*\bar \ell \\
 &= \cdots +(\xi_{-1}\ell +\xi_{-1}^*\bar \ell)\lambda^{-1}
 + \xi_{0}\ell+\xi_{0}^*\bar \ell
 + (\xi_{1}\ell+\xi_{1}^*\bar \ell )\lambda + \cdots,
 \end{align*}
 where we set 
\[
\ell = \frac12(1+\ip).	       
\]
 Since $\lambda \in \R^{+}$ corresponds to 
 $\lambda = \mu \ell + \mu^{-1} \bar \ell$ with $\mu \in \S^1_1 \;(\bar \mu =\mu^{-1})$ 
 and the properties of null basis $\{\ell, \bar \ell\}$, 
 that is, $\ell \bar \ell =0$ and 
 $\ell^2 =\ell$, $\bar \ell\,^2=\bar \ell$,
 we have 
\[
 \tilde \xi = \cdots + (\xi_{-1} \ell + \xi_{1}^* \bar \ell) \mu^{-1}
+ (\xi_{0} \ell + \xi_{0}^* \bar \ell) 
+ (\xi_{1} \ell + \xi_{-1}^* \bar \ell) \mu
+ \cdots.
\]
 Thus the following map is an isomorphism between $\lslR$ and
 $\lisu$
\begin{equation}\label{eq:funnyisom} 
 \lslR \ni \xi (\lambda) \mapsto 
 \xi (\mu)\ell + \xi^* (1/\bar \mu)\bar \ell
 \in \lisu,
\end{equation}
 where $\mu = \lambda \ell  + \lambda^{-1} \bar \ell$.
 
 Then combining two isomorphisms \eqref{eq:doubleiso} and \eqref{eq:funnyisom}, 
 we have isomorphisms 
\[
 \lslR \oplus\lslR \cong \lisu \oplus\lisu \cong \lslc,
\]
 where the maps are explicitly given by 
\begin{equation}\label{eq:isomap1}
 (\xi(\lambda), \eta(\lambda)) \mapsto 
 \left( \xi(\mu)\ell+
 \xi^*(1/\bar \mu)\bar \ell, \, 
\eta(\lambda)\ell+
 \eta^*(1/\bar \mu)\bar \ell
 \right)
\end{equation}
 for $ \lslR \oplus\lslR \cong \lisu \oplus\lisu$, and 
\begin{equation}\label{eq:isomap}
 (\xi(\lambda), \eta(\lambda)) \mapsto \xi(\mu)\ell+\eta^*(1/\bar \mu)\bar \ell
\end{equation}
 for  $ \lslR \oplus\lslR \cong \lslc$.
 Moreover, by the map \eqref{eq:isomap}, the following isomorphisms follow: 
\[
   \Lambda^{+} \mathfrak {sl}_2 \R \oplus 
   \Lambda^{-} \mathfrak {sl}_2 \R  \cong \lslcp,
 \quad 
   \Lambda^{-} \mathfrak {sl}_2 \R \oplus 
   \Lambda^{+} \mathfrak {sl}_2 \R  \cong \lslcn.
\]
 It is well known that \cite[Section 2.1]{DIT} the loop algebra 
 $\lslR$ 
 is a Banach Lie algebra and thus $\lslc (\cong  \lslR 
 \times  \lslR)$ is
 also a Banach Lie algebra, and the corresponding loop groups $\LSLR$ and
 $\LSLC(\cong \LSLR \times \LSLR)$ 
 become Banach Lie groups, respectively.

 Then the Birkhoff and Iwasawa decompositions 
 of 
 $\LSLR$ and $\LSLR \times \LSLR$ were proved in  Theorem 2.2 and Theorem 2.5 in \cite{DIT}: 
 The following multiplication maps
\[
 \LSLRN \times \LSLRP  \rightarrow \LSLR ,\quad 
 \LSLRP \times \LSLRN  \rightarrow \LSLR ,
\]
 and 
\[
\Delta (\LSLR \times \LSLR) 
 \times 
\LSLRP \times \LSLRN \rightarrow 
 \LSLR \times \LSLR 
\]
 are diffeomorphisms onto the open dense subsets of $\LSLR$ 
 and $ \LSLR \times \LSLR$, respectively.
 Then these decomposition 
 theorems can be translated to the Birkhoff and Iwasawa decompositions for 
 $\LSLC$.
 This completes the proof.
\end{proof}
\begin{Remark}
 In this paper, we consider only the loop group of 
 a Lie group $G$ which is defined on the hyperbola $\S^1_1$ and 
 has the power series expansion. We have denoted such loop group 
 by the symbol $\Lambda G_{\sigma}$.
 However in \cite{DIT}, the authors considered the loop group 
 $\tilde \Lambda G_{\sigma}$ which was a space of continuous maps from $\R^+$ 
 and it can be analytically  continued to $\mathbb C^{\times}$, 
 that is, an element of $\tilde \Lambda G_{\sigma}$ has the power series expansion.  
 If an element of $\tilde \Lambda G_{\sigma}$ is restricted to $\R^{+}$, then
 it corresponds to an element of $\Lambda G_{\sigma}$ as discussed above.
\end{Remark}
 In the following, we assume that an extended frame $F^{\mu}$ is in the
 big cell of $\LSLC$. 
 Using the Birkhoff decomposition in Theorem \ref{thm:BIdecomposition}, 
 we have the para-holomorphic data from a timelike minimal surface. 
\begin{Theorem}[The normalized potential]\label{thm:normalized}
 Let $F^{\mu}$ be an extended frame of a timelike minimal surface $f$ in 
 $\Nil$, and apply the Birkhoff decomposition in Theorem $\ref{thm:BIdecomposition}$ as
 $F^{\mu} = F_-^{\mu} F_+^{\mu}$ with $F_-^{\mu} \in  
 \LSLCN$ and
 $F_+^{\mu} \in  \LSLCP$. 
 Then the Maurer-Cartan form of $F_-^{\mu}$, that is, $\xi = (F_-^{\mu})^{-1} d F_-^{\mu}$, is para-holomorphic with respect to $z$. Moreover, $\xi$ 
 has the following  explicit form$:$
\begin{equation}\label{eq:normalized}
 \xi = \mu^{-1} \begin{pmatrix} 0 & b(z)\\ - \frac{B(z)}{b(z)} & 0
		   \end{pmatrix} dz,  
\end{equation}
 where 
\[
 b (z) = - \frac{\ip}4 \frac{h^2(z,0)}{h(0,0)}.
\]
 The data $\xi$ is called the {\rm normalized potential} of a
 timelike minimal surface $f$.
\end{Theorem}
\begin{proof}
 Let $F^{\mu}$ be an extended frame of a timelike minimal surface $f$ in $\Nil$.
 Applying the Birkhoff decomposition \eqref{eq:Birkhoff}
 in Theorem \ref{thm:BIdecomposition}:
\[
 F^{\mu} =  F^{\mu}_- F^{\mu}_+ \in 
 \Lambda^-_* {\rm SL}_{2} \C_{\sigma}\times \LSLCP.
\]
 Then the Maurer-Cartan form of $ F^{\mu}_{-}$ can be computed as
\begin{align}
\label{eq:compxi} \xi & = (F^{\mu}_{-})^{-1} d F^{\mu}_{-}  \\
\nonumber     & = F_+^{\mu} (F^{\mu})^{-1} d \left\{F^{\mu} (F_+^{\mu})^{-1} \right\}    \\
\nonumber     & =  F_+^{\mu} \alpha (F_+^{\mu})^{-1}  -  d F_+^{\mu} (F_+^{\mu})^{-1}.\end{align}
 Since $\xi$ takes values in $\lslc$ and does not have 
 $\mu^0$-term, 
 thus 
\[
 \xi = \mu^{-1}F_{+0} \begin{pmatrix} 0 &- \frac{\ip}4 h \\
 \frac{4 \ip B}{h} & 0
 \end{pmatrix} F_{+0}^{-1}\Big|_{\bar z=0} dz,
\]
 where $F_{+0}^{\mu}$ denotes the first coefficient of $F_{+}^{\mu}$ 
 expansion with respect to $\mu$, that is, 
 $F_{+}^{\mu} = F_{+0} + F_{+1} \mu + F_{+2} \mu^2+ \cdots$.
 Therefore $F_-^{\mu}$ is para-holomorphic with respect to $z$ and
 moreover, $\xi$ can be computed as
\[
 \xi(z, \mu) = \mu^{-1} \begin{pmatrix}
0 & - \frac{\ip}{4} h(z, 0) f_{0}^2(z, 0)  \\
\frac{4\ip B(z)}{h(z,0)}f_{0}^{-2}(z, 0) & 0
 \end{pmatrix}dz, 
\]
 where  $F_{+0}(z, 0)= \di(f_0(z, 0), f_0^{-1}(z, 0))$.
 We now look at the $\mu^0$-terms of both sides of \eqref{eq:compxi}:
 Then 
\[ 
0 = (F_{+0} \alpha_0^{\prime} F_{+0}^{-1}- d F_{+0} F_{+0}^{-1})|_{\bar z=0},
\]
 where $\alpha_0^{\prime}$ is
 $\alpha_0^{\prime}= (\frac12 \log h_{z}(z, 0) ) \sigma_3 dz$.
 It is equivalent to $ d F_{+0} = F_{+0} \alpha_0^{\prime}$, and therefore 
 \[
  f_{0}(z, 0) = h^{1/2}(z, 0) c,
 \]
 where $c$ is some constant, follows. 
 Since $F_{+0}(0, 0)= \id$, thus $c= h^{-1/2}(0,0)$. 
 This completes the proof.
\end{proof}

\subsection{From para-holomorphic potentials to minimal surface: 
 The Iwasawa decomposition}
 Conversely, in the following theorem we will show 
 a construction of timelike minimal surface from 
 normalized potentials as defined in \eqref{eq:normalized},
 the so-called \textit{generalized Weierstrass type representation}.

 \begin{Theorem}[The generalized Weierstrass type representation
]\label{thm:Weierstrass}
  Let $\xi$ be a normalized potential defined in \eqref{eq:normalized},
  and let $F_{-}$ be the solution of 
\[
 \partial_z F_{-} = F_{-} \xi, \quad F_{-}(z=0)= \id.
\]
 Then applying the Iwasawa decomposition in Theorem $\ref{thm:BIdecomposition}$
 to $F_-$, that is $F_- = F^{\mu} V_+$ with 
 $F^{\l} \in \LISU$
 and $V_+ \in \LSLCP$, and choosing a proper diagonal ${\rm U}_1^{\prime}$-element $k$, 
 $F^{\mu}k$ is 
 an extended frame of the normal Gauss map $f^{-1} N$ of a 
 timelike minimal surface $f$ in $\Nil$ up to the change of coordinates. 
 \end{Theorem}
\begin{proof}
 It is easy to see that the solution $F_-$ takes values in 
 $\LSLC$. Then apply the Iwasawa decomposition to $F_-$ (on the big cell), that is,
\[
 F_- = F^{\mu} V_+ \in \LISU
 \times
\LSLCP.
\]
 We now compute the Maurer-Cartan form of $F^{\mu}$ as $(F^{\mu})^{-1} d F^{\mu}$,
\begin{align}
\label{eq:MCF} \alpha^{\mu}& =(F^{\mu})^{-1} d F^{\mu} \\
 &= V_+ F_-^{-1} d (F_- V_+^{-1})\\
\nonumber &= V_+ \xi V_+^{-1} - d V_+ V_+^{-1}.
\end{align}
 From the right-hand side of the above equation, it is easy to see 
 $\alpha^{\mu} = \mu^{-1}\alpha_{-1}+ \alpha_{0} + \mu^{1}\alpha_{1}+ \cdots$.
 Since $F^{\mu}$ takes values in $\LISU$, thus 
 \[
  \alpha^{\mu} = \mu^{-1}\alpha_{-1}+ \alpha_{0} + \mu^{1}\alpha_{1},
 \]
 and $\alpha_{j}^* = \alpha_{-j}$ holds. From the form of $\xi$
 and the right-hand side of \eqref{eq:MCF},  the Maurer-Cartan form 
 $\alpha^{\mu}$ almost has the form in \eqref{eq:alphamu}.
 Finally a proper choice of a diagonal $\Uone$-element $k$ and a change of coordinates 
 imply that $\alpha^{\mu}$ is exactly the same form in \eqref{eq:alphamu}. 
 This completes the proof.
\end{proof}

\begin{Remark}
Taking an extended frame $\tilde{F}^\mu$ given by Theorem \ref{thm:Weierstrass} with a $\LISU$-valued initial condition :
$F_-(z=0)=A \:\:(A \in \LISU)$,
extended frames $\tilde{F}^\mu$ and $F^\mu$ differ by $A$,
that is,
$\tilde{F}^\mu = A F^\mu$.
In general, timelike minimal surfaces in $\Nil$ corresponding to extended frames for different initial conditions are not isometric.
\end{Remark}

\section{Examples}\label{sc:Example}
 In this section we will give some examples of timelike minimal surfaces in $\Nil$ 
 in terms of para-holomorphic potentials and the generalized Weierstrass type 
 representation as explained in the previous section. 
\subsection{Hyperbolic paraboloids corresponding to circular cylinders}
 Let $\xi$ be the normalized potential defined as
\[\xi =
 \mu^{-1}
 \begin{pmatrix}
  0& -\frac {\ip} 4\\
  \frac {\ip} 4 &0
 \end{pmatrix}
 dz.
\]
The solution of the equation $d F_- = F_- \xi$
 with the initial condition $F_- (z=0) = \id$
 is given by
\[ F_- =
 \begin{pmatrix}
 \cos \frac {\mu^{-1} z} 4 & -\ip \sin \frac {\mu^{-1} z} 4 \\
 \ip \sin \frac {\mu^{-1} z} 4 & \cos \frac {\mu^{-1} z} 4 
 \end{pmatrix}.
\]
 Applying the Iwasawa decomposition to the solution $F_-$:
 \[ F_- = F^{\mu} V_+,\]
 we obtain an extended frame $F^\mu: \C \to  \LISU $:
\[
 F^\mu =
 \begin{pmatrix}
 \cos \frac{ {\mu}^{-1} z + \mu \bar z } 4 & -\ip \sin \frac{ {\mu}^{-1} z + \mu \bar z } 4\\
 \ip \sin \frac{ {\mu}^{-1} z + \mu \bar z } 4 & \cos \frac{ {\mu}^{-1} z + \mu \bar z } 4
 \end{pmatrix}.
\]
 Then, by Theorem \ref{thm:Sym2}, we have the map $f_{\Min}$ explicitly
\[
 f_{\Min} = \frac{1}{2}
\begin{pmatrix}
 -\ip \cos \frac{\mu^{-1}z + \mu \bar z}2 & -\sin \frac{\mu^{-1}z + \mu \bar z}2 - \frac{\mu^{-1}z - \mu \bar z}2\\
 -\sin \frac{\mu^{-1}z + \mu \bar z}2 + \frac{\mu^{-1}z - \mu \bar z}2& \ip \cos \frac{\mu^{-1}z + \mu \bar z}2
\end{pmatrix},
\]
and
\[
 \hat{f} =\frac12
 \begin{pmatrix}
 -\frac{-\mu^{-1}z +\mu \bar z}4 \sin \frac{\mu^{-1}z +\mu \bar z}2 &
 -\sin \frac{\mu^{-1}z +\mu \bar z}2 -\frac{\mu^{-1}z -\mu \bar z}2\\
 -\sin \frac{\mu^{-1}z +\mu \bar z}2 +\frac{\mu^{-1}z -\mu \bar z}2 &
 \frac{-\mu^{-1}z +\mu \bar z}4 \sin \frac{\mu^{-1}z +\mu \bar z}2
 \end{pmatrix}.
\]
 Thus we obtain timelike surfaces $f_{\Min}$ with the constant mean curvature $1/2$ in $\Min$
 and timelike minimal surfaces $f^\mu$ in $\Nil$:
\[
  f_{\Min} =\left(
 \sin \frac{\mu^{-1}z + \mu \bar z}2,
 \ip \frac{\mu^{-1}z - \mu \bar z}2,
 -\cos \frac{\mu^{-1}z + \mu \bar z}2
 \right)
\]
 and
\[
 f^\mu=\left(
 \ip \frac{\mu^{-1}z - \mu \bar z}2,
 \sin \frac{\mu^{-1}z + \mu \bar z}2,
 \ip \frac{\mu^{-1}z - \mu \bar z}4 \sin \frac{\mu^{-1}z + \mu \bar z}2
 \right)
\]
 for $\mu = e^{\ip t}$ with sufficiently small $t$
 on some simply connected domain $\D$.
 Each surface $f^\mu$ describes a part of a hyperbolic paraboloid $x_3= x_1 x_2/2$.
 Furthermore $f^\mu$ has the function $h= 1$,
 the Abresch-Rosenberg differential $B^\mu dz^2 = \mu^{-2}/16 dz^2$ on $\D$
 and
 the first fundamental form $I$ of $f^\mu$ is
 $I= \cos ^2 (\mu^{-1} z + \mu \bar z)/2 dz d\bar z$.
 The corresponding timelike CMC $1/2$ surfaces $f_{\Min}$
 are called circular cylinders.

\begin{figure}[t]
\centering
\includegraphics[width=5cm]{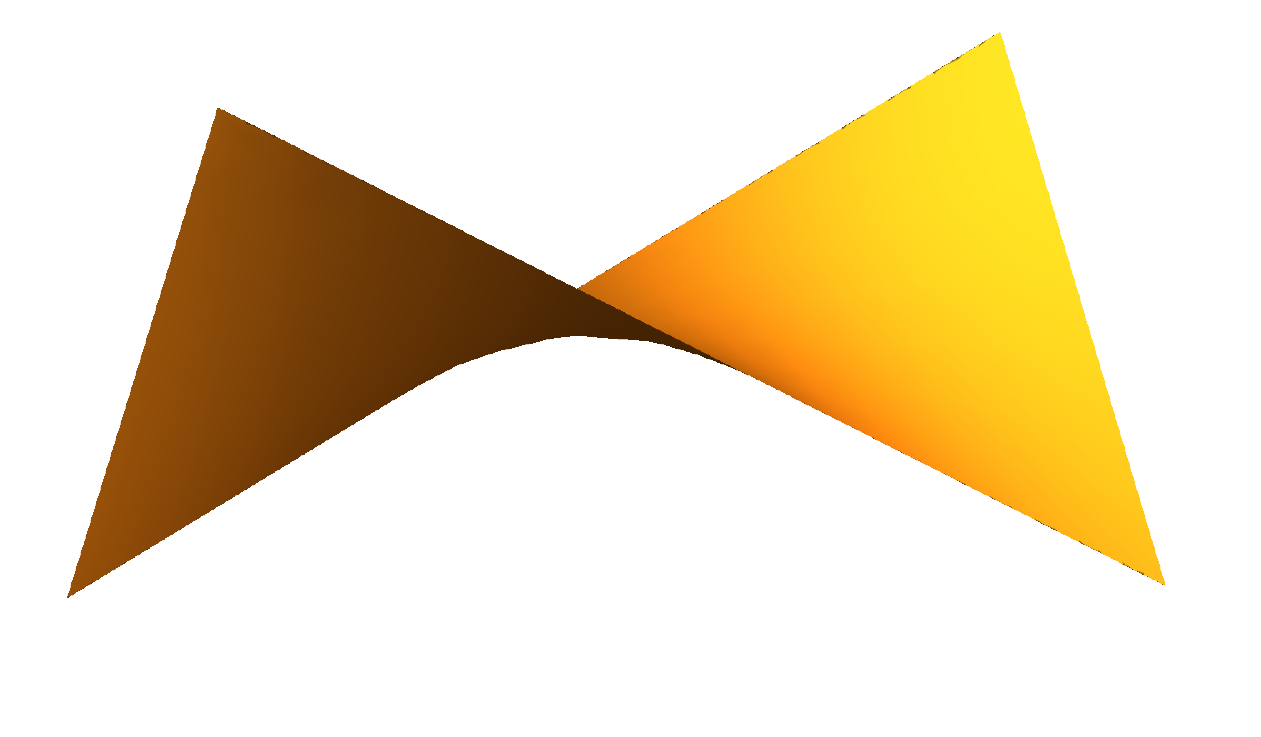}\hspace{0.5cm}
\includegraphics[width=5cm]{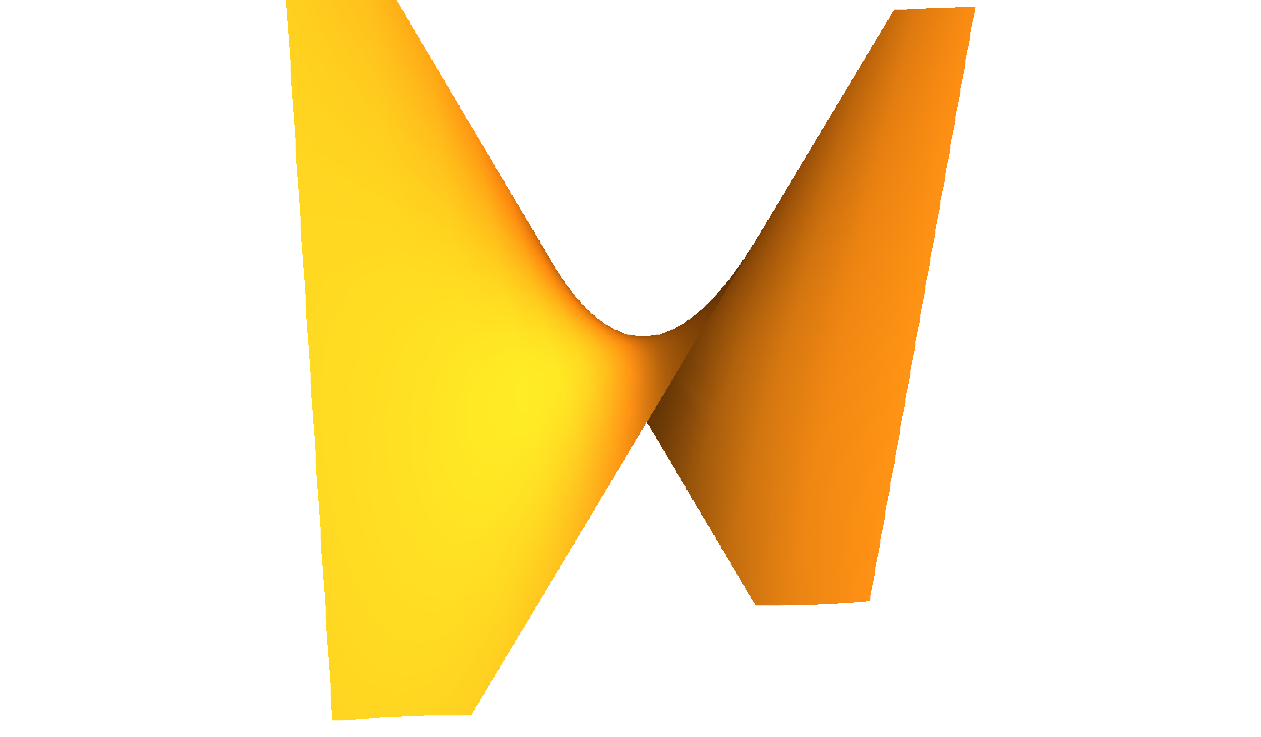}
\caption{Hyperbolic paraboloids corresponding to a cylinder (left) and 
 a hyperbolic cylinder (right).}\label{fig:cylinder}
\end{figure}

\subsection{Hyperbolic paraboloids corresponding to hyperbolic cylinders}
 Define the normalized potential $\xi$ as
\[
 \xi = \mu^{-1}
 \begin{pmatrix}
 0 & - \frac \ip 4 \\
 - \frac \ip 4 & 0
 \end{pmatrix} dz.
\]
 The solution of the equation $d F_- = F_- \xi$
 with the initial condition $F_- (z=0) = \id$
 is given by
\[F_- =
 \begin{pmatrix}
 \cosh \frac {\mu^{-1} z} 4 & -\ip \sinh \frac {\mu^{-1} z} 4 \\
 -\ip \sinh \frac {\mu^{-1} z} 4 & \cosh \frac {\mu^{-1} z} 4 
 \end{pmatrix}.
\]
 Applying the Iwasawa decomposition to the solution $F_-$:
 \[ F_- = F^{\mu} V_+,\]
 we obtain an extended frame $F^\mu: \C \to  \LISU $:
\[
 F^\mu =
 \begin{pmatrix}
 \cosh \frac{ -{\mu}^{-1} z + \mu \bar z } 4 & \ip \sinh \frac{ -{\mu}^{-1} z + \mu \bar z } 4\\
 \ip \sinh \frac{ -{\mu}^{-1} z + \mu \bar z } 4 & \cosh \frac{ -{\mu}^{-1} z + \mu \bar z } 4
 \end{pmatrix}.
\]
 Then, by Theorem \ref{thm:Sym2}, we have the map $f_{\Min}$ for $F^\mu$ explicitly
\[
 f_{\Min} = \frac12
 \begin{pmatrix}
 -\ip \cosh \frac{-\mu^{-1}z + \mu \bar z}2 &
 -\frac{\mu^{-1}z + \mu \bar z}2 + \ip \sinh \ip \frac{-\mu^{-1}z + \mu \bar z}2\\
 -\frac{\mu^{-1}z + \mu \bar z}2 - \ip \sinh \ip \frac{-\mu^{-1}z + \mu \bar z}2&
 \ip \cosh \frac{-\mu^{-1}z + \mu \bar z}2
 \end{pmatrix},
\]
 and thus we obtain timelike surfaces $f_{\Min}$ with the constant mean curvature $1/2$ in $\Min$
 and timelike minimal surfaces $f^\mu$ in $\Nil$:
\[
  f_{\Min} =\left(
 \frac{ \mu^{-1} z + \mu \bar z}2,
 - \sinh \ip \frac{ -\mu^{-1} z + \mu \bar z}2,
 -\cosh \frac{ -\mu^{-1}z + \mu \bar z}2
 \right)
\]
 and
\[
 f^\mu=\left(
 -\sinh \ip \frac{-\mu^{-1} z + \mu \bar z}2,
 \frac{\mu^{-1} z + \mu \bar z}2,
 \frac{\mu^{-1}z + \mu \bar z}4 \sinh \ip \frac{ -\mu^{-1}z + \mu \bar z}2
 \right)
\]
 for any $\mu$ on $\C$.
 Each timelike minimal surface $f^\mu$ describes the hyperbolic paraboloid $x_3=-x_1 x_2/2$
 and has the function $h=1$,
 the Abresch-Rosenberg differential $B^\mu dz^2 = -\mu^{-2}/16 dz^2$
 and the first fundamental form $I(\mu) = \left\{\cosh  \tfrac{\ip}2 ( -\mu^{-1} z + \mu \bar z)\right\}^2   dz d \bar z$.
 The corresponding timelike CMC $1/2$ surfaces $f_{\Min}$
 are called hyperbolic cylinders.
 
\subsection{Horizontal plane}
 Let $\xi$ be the normalized potential defined by
\[
 \xi = \mu^{-1}
 \begin{pmatrix}
 0 & -\ip\\
 0 & 0
 \end{pmatrix}dz.
\]
 The solution of the equation $d F_- = F_- \xi$ under the initial condition $F_-(z=0) = \id$ is given by
\[
 F_- =
 \begin{pmatrix}
 1 & -\ip \mu^{-1} z\\
 0 & 1
 \end{pmatrix}.
\]
 Then by the Iwasawa decomposition of the solution $F_- = F^\mu V_+$,
 we have an extended frame $F^\mu: \tilde{\D} \to \LISU$:
\begin{equation}\label{eq:exthori}
  F^\mu =
 \frac{1} {( 1+ z \bar z)^{1/2}}
 \begin{pmatrix}
 1 & -\ip \mu^{-1} z\\
 \ip \mu \bar z & 1
 \end{pmatrix},
\end{equation}
 where $\tilde{\D}$ is a simply connected domain defined as $\tilde{\D} = \left\{ z \in \C | z \bar z>-1 \right\}$.
 Then $f_{\Min}$ is given by
\[
 f_{\Min} =\frac1 {1 + z \bar z}
 \begin{pmatrix}
 \frac {\ip(3z \bar z -1)} 2 & -2 \mu^{-1} z\\
 -2\mu \bar z & -\frac {\ip (3z \bar z -1)} 2
 \end{pmatrix}.
\]
 Hence the timelike surfaces $f_{\Min}$ with the constant mean curvature $1/2$ in $\Min$
 and the timelike minimal surfaces $f^\mu$ in $\Nil$ are computed as
\[
 f_{\Min} =\left(
 \frac{2     ( \mu^{-1} z + \mu \bar z)}{1 + z \bar z},
 \frac{2 \ip ( \mu^{-1} z - \mu \bar z)}{1 + z \bar z},
 \frac{3 z \bar z - 1}{1 + z \bar z}
 \right)
\]
 and
\[
 f^\mu = \left(
 \frac{2 \ip (\mu^{-1} z -\mu \bar z)}{1 + z \bar z}, \:
 \frac{2(\mu^{-1} z +\mu \bar z)}{1 + z \bar z}, \:
 0
 \right).
\]
 The surfaces $f^\mu$ are defined on $D= \left\{ z \in \C | -1<z \bar z <1 \right\}$.
 In fact the first fundamental form $I$ of $f^\mu$ is computed as
\[
 I= 16 \frac {( 1-z \bar z )^2 } {(1 + z \bar z)^4} dz d \bar z.
\]
 Moreover the Abresch-Rosenberg differential $B^\mu dz^2$ vanishes on $D$.
 
 In general the graph of the function $F(x_1,x_2) = ax_1 + bx_2 +c$ for $a,b,c \in \R$
 describes a timelike minimal surface
 on $\D = \left\{(x_1,x_2) \mid -(a + x_2/2)^2 + (b - x_1/2)^2 +1 >0 \right\}$.
 This plane has positive Gaussian curvature $K$:
\[
 K= \frac {2 ( - ( a+\tfrac{1}{2}x_2 )^2 + ( b - \tfrac{1}{2}x_1 )^2 +1) +1}
 {4 ( -( a+\tfrac{1}{2}x_2 )^2 + ( b - \tfrac{1}{2}x_1)^2 +1)^2},
\]
 and it will be called the \textit{horizontal umbrellas}. The horizontal 
umbrellas are obtained by different choices of initial conditions of the extended
 frame of $F^{\mu}$ in \eqref{eq:exthori}. For examples the extended frame 
 $F_0 F^{\mu}$ with 
\[
 F_0= \begin{pmatrix} \cosh a & \mu^{-3}\sinh a\\
 \mu^3 \sinh a & \cosh a \end{pmatrix} \in \LISU,
\] 
 where $a \in \mathbb R$ gives a horizontal umbrella which is 
 not a horizontal plane. 
 
\subsection{B-scroll type minimal surfaces}\label{sbsc:Bscroll}
 Let $\xi$ be a normalized potential defined as 
\[
\xi = \mu^{-1} \begin{pmatrix} 0 & - \frac{\ip}4 \\  -S(z)\bar \ell  &0
\end{pmatrix} dz,
\]
 where  $\bar \ell =  \frac12 (1 -
 \ip)$ and $S(z)$ is a para-holomorphic function.
 The solution $\Phi$ of $d \Phi = \Phi \xi$ with $\Phi(z=0) = \id$ cannot be 
 computed explicitly, but it can be partially computed as follows:
 It is known that a para-holomorphic function $S(z)$ can be expanded
 as
\[
 S(z) = Q(s) \ell + R(t) \bar \ell
\]
 with  $ s= x + y$ and $t = x- y$ for para-complex coordinates $z = x + \ip y$,
 $Q = \Re S + \Im S$  and $R = \Re S - \Im S$.
 Note that $\ell^2 = \ell, \bar \ell^2 = \ell, \ell \bar \ell =0$.
 Moreover, from the definition of $s$ and $t$, 
 $d z = \ell d s + \bar \ell dt$ follows.
 Then the para-holomorphic potential $\xi$ can be decomposed by 
 \[
  \xi = \xi^s \ell + {\xi^t}^* \bar \ell,
 \]
 with ${\xi^t}^* = - \sigma_3 \left(\overline {\xi^t(1/\bar \mu)}\right)^{T} \sigma_3$ 
 and 
\begin{equation}\label{eq:xist}
\xi^s = \l^{-1}\begin{pmatrix}
0 &  - \frac14 \\  0 & 0
\end{pmatrix} ds, 
\quad 
\xi^t = \l \begin{pmatrix}
0 & - R(t) \\  \frac14 & 0
\end{pmatrix} dt.
\end{equation}
 Then by the isomorphism in \eqref{eq:funnyisom}, 
 the pair $(\xi^s(\lambda), \xi^t(\lambda))$ is the 
 normalized potential in \cite[Section 6.2]{DIT} for 
 a timelike CMC surface in $\Min$ \textit{B-scroll}.
 Then the solution of $d \Phi = \Phi \xi$  can be 
 computed by 
\[
  d \Phi^s = \Phi^s \xi^s, \quad   d \Phi^t = \Phi^t \xi^t\quad 
 \mbox{with}\quad \Phi^s (0)= \Phi^t (0)= \id
\]
 and $\Phi$ is given by 
 $\Phi = \Phi^s \ell + {\Phi^t}^* \bar \ell$, where 
 $\Phi^s = \Phi^s (\mu)$ and 
 ${\Phi^t}^* = \sigma_3\overline{\Phi^t (1 / \bar \mu)}^{T-1}\sigma_3$
 for $\Phi^s, \Phi^t \in \LSLR$.
 Then $\Phi^s$ can be explicitly integrated as 
\[
\Phi^s = \begin{pmatrix}1 & - \frac14 \l^{-1}s\\ 0 & 1
	 \end{pmatrix},
\]
 while $\Phi^t$ cannot be explicitly integrated. Set
\begin{equation}\label{eq:Phit}
 \Phi^t= \id + \sum_{k\geq 1} \l^{k} 
\begin{pmatrix} a_k & b_k\\ c_k & d_k \end{pmatrix},
\end{equation}
 where $a_{2k+1} = d_{2k+1} = b_{2k} = c_{2k}=0$ for all $k \geq 1$.
 Then applying the Iwasawa decomposition in Theorem \ref{thm:BIdecomposition} 
 to $\Phi$, that is, $\Phi = F^{\mu} V_+$, one can 
 compute
\[
 \Phi = \Phi^s \ell + {\Phi^t}^* \bar \ell 
 = (\hat F \ell + \hat F^* \bar \ell )(\hat V_+ \ell + \hat V_-^* \bar \ell),
\] 
 where $F^{\mu}= \hat F \ell + \hat F^* \bar \ell$ and 
$V_+ = \hat V_+\ell + \hat V_-^*  \bar \ell$ and $\hat F \in \LSLR$,
 $\hat V_{+} \in \LSLRP$ and $\hat V_{-} \in \LSLRN$. Note that it is
 equivalent to the Iwasawa decomposition of $\LSLR \times \LSLR$, that is,
\begin{equation}\label{eq:Iwasawaphist}
 (\Phi^s, \Phi^t) = (\hat F, \hat F) (\hat V_+, \hat V_-).
\end{equation}
\begin{Proposition}
 The map $\hat F$ can be computed as follows:
\begin{equation}\label{eq:hatF}
 \hat F = \Phi^t \Phi_{-}, \quad \mbox{with}
 \quad \Phi_{-} = \begin{pmatrix} \left(1+ \frac14 sc_1\right)^{-1} &  -\frac14 \l^{-1}s    \\ 
0 & 1+ \frac14 sc_1
 \end{pmatrix},
\end{equation}
 where $c_1= c_1(s, t)$ is the function defined in \eqref{eq:Phit}.
\end{Proposition}
\begin{proof}
 From \eqref{eq:Iwasawaphist}, the map $\hat F$ can be computed as
 \[
  {\Phi^s}^{-1} \Phi^t = \hat V_{+}^{-1} \hat V_-
 \]
 by the Birkhoff decomposition of $\Phi_s^{-1} \Phi^t$ and set 
 $\hat F = \Phi^t \hat V_-^{-1} = \Phi^s \hat V_+^{-1}$. 
 We then multiply $\Phi_-$ on $ {\Phi^s}^{-1} \Phi^t $
 by right, and a straightforward computation shows that 
\begin{align*}
  {\Phi^s}^{-1} \Phi^t \Phi_-&=  \begin{pmatrix}1 &  \frac14 \l^{-1}s\\ 0 & 1 \\ 
	 \end{pmatrix}\left(\id + \sum_{k \geq 1} \l^k 
 \begin{pmatrix} a_k & b_k \\ c_k & d_k \end{pmatrix}\right)
\Phi_- \\
 & = \left\{\begin{pmatrix} 1 + \frac14 sc_1 &  \frac14 \l^{-1} s \\
 0 & 1
     \end{pmatrix} 
 + O(\l)\right\} 
 \begin{pmatrix} 
 \frac1{1+  \frac14 s c_1} &- \frac14\l^{-1}  s   \\ 0&  1 + \frac14 s c_1 
 \end{pmatrix} \\
 & = \id + O(\l)
\end{align*}
 holds.
 Therefore ${\Phi^s}^{-1} \Phi^t \Phi_- \in \LSLRP$ with identity at $\lambda =0$, and
 ${\Phi^s}^{-1} \Phi^t  = \hat V_+^{-1} \Phi_-^{-1}$ is the Birkhoff decomposition.
 This completes the proof.
\end{proof}
 Plugging the $F^{\mu}$ into $f_{\Min}$ in \eqref{eq:SymMin}, we obtain 
\[
 f_{\Min} = \left\{\gamma(t)+ q(s, t) B(t)\right\} \ell 
 + \left\{\gamma(t)+ q(s, t) B(t)\right\}^* \bar \ell,
\]
 where $A^*= - \sigma_3 \overline{A(1/ \bar \mu)}^T\sigma_3$
 for $A \in \lslR$ and
\begin{align*}
 \gamma(t)&= -\ip \mu (\partial_\mu \Phi^{t})(\Phi^{t})^{-1}
 - \frac \ip2 \ad (\Phi^t) \sigma_3 ,\\
 B(t)&= -\ip \mu \ad \Phi^t \begin{pmatrix}0 & 1 \\ 0 & 0 \end{pmatrix},
 \\
 q(s, t) &= \frac{s}{2 (1+ \frac1{16} s t)}.
\end{align*}
 Under the new coordinates $(q, t)$, $f_{\Min}$ is a so-called \textit{B-scroll}, 
 that is, $\gamma$ is null curve in $\Min$ and $B$ is the bi-normal null 
 vector of $\gamma$, see in detail \cite[Section 6.2]{DIT}. 

 Further plugging the $F^{\mu}$ into $\hat f$ in \eqref{eq:symNil},
 we obtain 
\[
 \hat f = \left\{\hat \gamma(t)+ q(s, t) \hat B(t)\right\} \ell 
 + \left\{\hat \gamma(t)+ q(s, t) \hat B(t)\right\}^* \bar \ell,
\]
 where 
\[
 \hat \gamma(t) = \gamma(t)^{o} -  \frac{\ip}{2} \mu \partial_{\mu} \gamma(t)^d, 
\quad \mbox{and}\quad
 \hat B(t) = B(t)^{o} -  \frac{\ip}{2} \mu \partial_{\mu} B(t)^d.
\]
 A straightforward computation shows that $\exp( \hat \gamma(t) \ell + \hat \gamma(t)^* \bar \ell )$ is 
 null curve in $\Nil$ and $\hat B(t) \ell + \hat B(t) ^* \bar \ell $ is a bi-normal vector of $\exp( \hat \gamma(t) \ell + \hat \gamma(t)^* \bar \ell )$ 
 analogous to the Minkowski case. Therefore we call $f^{\mu}$ is 
 the \textit{B-scroll type} minimal surface in $\Nil$.
 We will investigate property of  the B-scroll type minimal surface 
 in a separate publication.

\appendix

\section{Timelike constant mean curvature surfaces in $\mathbb E^3_1$}\label{timelikemin}
 We recall the geometry of timelike surfaces in Minkowski 3-space.
 Let $\Min$ be the Minkowski 3-space with the Lorentzian metric 
\[
 \langle \cdot,\cdot \rangle
 =
 dx_1^2 - dx_2^2 +dx_3^2,
\]
 where $(x_1, x_2, x_3)$ is the canonical coordinate of $\R^3$.
 We consider a conformal immersion $\varphi : M \to \Min$
 of a Lorentz surface $M$ into $\Min$.
 Take a para-complex coordinate $z = x + \ip y$  
 and represent the induced metric by $e^u dz d\bar z$.

 Let $N$ be the unit normal vector field of $\varphi$.
 The second fundamental form $I\hspace{-1pt}I$ of $\varphi$ derived from $N$ is defined by
\[
 I\hspace{-1pt}I = -\langle d\varphi, dN \rangle.
\]
 The mean curvature $H$ of $\varphi$ is defined by
\[
 H = \frac12 {\rm tr}(I\hspace{-1pt}I \cdot I ^{-1}).
\]

 For a conformal immersion $\varphi : \D \to \Min$, 
 define para-complex valued functions $\phi_1, \phi_2, \phi_3$ by  
\[
 \varphi_z = (\phi_2, \phi_1, \phi_3).
\]
 The analogy of the discussion in Section \ref{sec:deracequation}
 shows that there exists
 $\epsilon \in \{\pm \ip \}$
 and a pair of para-complex functions $(\psi_1, \psi_2)$ such that
\[
  \phi_1 = \epsilon \left((\overline{\psi_2})^2 +(\psi_1)^2\right), \quad
 \phi_2 = \epsilon \ip \left((\overline{\psi_2})^2 -(\psi_1)^2\right), \quad
 \phi_3 =  2 \ip \psi_1 \overline{\psi_2}.
\]

 Then the conformal factor $e^u$ and the unit normal vector field $N$ of $\varphi$ can be represented as 
\[
 e^u = 4(\psi_2 \overline{\psi_2} - \psi_1 \overline{\psi_1})^2,
\]
\begin{equation}\label{normal}
 N = \frac1{\psi_2 \overline{\psi_2} - \psi_1 \overline{\psi_1}} 
  \left( 
  -\epsilon (\psi_1\psi_2 - \overline{\psi_1}\overline{\psi_2}),\,
  \epsilon \ip (\psi_1\psi_2 + \overline{\psi_1}\overline{\psi_2}),\,
  \psi_2 \overline{\psi_2} + \psi_1 \overline{\psi_1}
  \right).
\end{equation}

As well as timelike surfaces in $\Nil$, we can show that
$(\psi_1, \psi_2)$ is a solution of the nonlinear Dirac equation for a timelike surface in $\Min$:
\begin{equation}\label{diraceq}
\left(
\begin{array}{ll} 
 (\psi_2)_z + \mathcal{U} \psi_1\\
 -(\psi_1)_{\bar z} + \mathcal{V} \psi_2
\end{array}
\right)
=
\left(
\begin{array}{ll}
 0\\0
\end{array}
\right)
\end{equation}
 where $\mathcal{U} =\mathcal{V}= \ip H \hat \epsilon e^{u/2} /2$ 
 and $\hat \epsilon$ is the sign of $\psi_2 \overline{\psi_2} - \psi_1 \overline{\psi_1}$.
 Conversely, if a pair of para-complex functions $(\psi_1, \psi_2)$
 satisfying the nonlinear Dirac equation \eqref{diraceq} and
 $\psi_2 \overline{\psi_2} - \psi_1 \overline{\psi_1} \neq0$ is given,
 there exists a conformal timelike surface in $\Min$
 with the conformal factor
 $e^u = 4(\psi_2 \overline{\psi_2} - \psi_1 \overline{\psi_1})^2$.
\begin{Theorem}\label{minws}
 Let $\D$ be a simply connected domain in $\C$,
 $\mathcal U$ a purely imaginary valued function
 and the vector $(\psi_1, \psi_2)$ a solution of the nonlinear Dirac equation \eqref{diraceq}
 satisfying $\psi_2 \overline{\psi_2} - \psi_1 \overline{\psi_1} \neq0$.
 Take points $z_0 \in \D$ and $f(z_0) \in \Min$,
 set $\epsilon$ as either $\ip$ or $-\ip$
 and define a map $\Phi$ by
\[
 \Phi = \left( \epsilon \ip \left( (\overline{\psi_2})^2 - (\psi_1)^2\right),\,
 \epsilon \left( (\overline{\psi_2})^2 + (\psi_1)^2\right), \,
 2 \ip \psi_1 \overline{\psi_2} \right).
\]
 Then the map $f : \D \to \Min$ defined by
\begin{equation}\label{eq:minws}
 f(z) := f(z_0) + 
 \Re \left(
 \int_{z_0}^z \Phi dz
 \right)
\end{equation}
 describes a timelike surface in $\Min$. 
\end{Theorem}

\begin{proof}
 A straightforward computation shows that the $1$-form $\Phi dz + \overline{\Phi} d\bar{z}$ is a closed form.
 Then Green's theorem implies that $f(z)$ is well-defined.
 Thus we have $f_z = \Phi$.
 By setting $\phi_k \:(k=1,2,3)$ as $\Phi = (\phi_2, \phi_1, \phi_3)$,
 we derive $\phi_2^2 -\phi_1^2 +\phi_3^2 =0$
 and $\phi_2 \overline{\phi_2} -\phi_1 \overline{\phi_1} + \phi_3 \overline{\phi_3}
 = 2(\psi_2 \overline{\psi_2} -\psi_1 \overline{\psi_1})^2$.
 This means that $f$ is conformal, and then timelike.
\end{proof}
\begin{Remark}
 The timelike surface defined in Theorem\ref{minws} is conformal with respect to the coordinate $z$.
 Denoting the mean curvature by $H$ 
 and the conformal factor by $e^u$ 
 then we have $\mathcal{U} =\ip H \hat \epsilon e^{u/2} /2$
 where $\hat \epsilon $ is the sign of $\psi_2 \overline{\psi_2} - \psi_1 \overline{\psi_1}$.
\end{Remark}
 Obviously, the Dirac equation for timelike minimal surfaces in $(\Nil, ds_-^2)$ 
 coincides the one for timelike constant mean curvature $H=1/2$ surfaces in $\Min$.
 Combining the identification of $\isu$ with $\Min$ and \eqref{eq:f_z},
 we can show that
 the corresponding timelike constant mean curvature $1/2$ surfaces
 for timelike minimal surfaces $f^{\mu}$ in $(\Nil, ds_-^2)$ are given by $f_{\Min}$ up to translations
 and represented in the form of
 \eqref{eq:minws}.
 It is easy to see that the unit normal vector field \eqref{normal}
 of the timelike surface $f_{\Min}$ can be written as $N_{\Min}$ by the identification of $\isu$ and $\Min$.  

\section{Without para-complex coordinates}\label{sc:nopara}
 As we have explained in Example \ref{sbsc:Bscroll}, the normalized 
 potential $\xi$ which is a $1$-form taking values in $\lslc$ 
 can be translated to the pair of two real potentials which is
 a pair of $1$-forms taking values in $\lslR \times \lslR$. 
 It can be generalized to 
 any normalized potential $\xi$ any pair of two real potentials 
 $(\xi^s, \xi^t)$ as follows:
 For a normalized potential 
\[
 \xi = \mu^{-1} \begin{pmatrix}
 0 & -\frac{\ip}4 b(z) \\ 4 \ip \frac{B(z)}{b(z)} & 0
\end{pmatrix} dz,
\]
 where  $b(z)= h^2(z, 0)h^{-1}(0, 0)$, one can define a pair of $1$-forms
 by $\xi = \xi^s \ell + {\xi^t}^* \bar \ell$ such that
 \[
  \xi^s = \lambda^{-1} \begin{pmatrix} 0 & - \frac14 f(s) \\Q (s)/f(s) &0  \end{pmatrix}
  ds,
  \quad
  \xi^t = \lambda \begin{pmatrix} 0 &- R (t)/g(t) \\ \frac14 g(t) & 0 \end{pmatrix}
  dt,
 \]
 where para-complex coordinates $z = x + \ip y$ define null 
 coordinates $(s, t)$ by $x = s+ t$ and $y = s- t$, and
 the functions $f(s)$ and $g(t)$ are given by 
 \[
 f(s) = \Re b(s) + \Im b(s), \quad 
 g(t) = \Re b(t) - \Im b(t),
 \]
 and the functions $Q(s)$ and $R(t)$ are given by 
\begin{equation}\label{eq:QR}
 Q(s) = 4(\operatorname{Re} B(s)  + \operatorname{Im}B(s) ), \quad 
R(t) = 4( \operatorname{Re} B(t)  - \operatorname{Im}B(t) ).
\end{equation}
 Note that we use relations 
 $b(z) = f(s) \ell + g(t) \bar \ell$, and 
 $4 B(z) = Q(s) \ell + R(t) \bar \ell$
with $\ell =\frac{1+\ip}2$
 and $1/(f(s) \ell + g(t)\bar \ell) = \ell /f(s) + \bar \ell/ g(t)$.

 Again that the para-holomorphic solution $\Phi$ taking values in $\LSLC$ 
 of $d \Phi = \Phi \xi$ with $\Phi(0) =\id$ can 
 be identified with the pair $(\Phi^s, \Phi^t)$ by 
\[
 \Phi= \Phi^s\ell+  {\Phi^t}^* \bar \ell,
\]
 where $\Phi^s = \Phi^s(\mu)$ and 
 ${\Phi^t}^* = \sigma_3 \overline{\Phi^t(1/ \bar\mu)}^{T -1}\sigma_3$.
 Thus using the partial differentiations with respect to $s$ and $t$ by
\[
 \partial_s = \ell \partial_z + \bar \ell \partial_{\bar z}
 \quad {\rm and} \quad
 \partial_t = \bar \ell \partial_z + \ell \partial_{\bar z},
\]
 we need to consider the pair of ODEs 
\[
 \partial_s \Phi^s = \Phi^s \xi^s,  \quad \partial_t \Phi^t = \Phi^t \xi^t, 
\]
 with the initial condition $(\Phi^s(0), \Phi^t(0)) = (\id, \id)$.
 The Iwasawa decomposition of $\Phi$, that is $\Phi = F^{\mu} V_+$, can 
 be again translated to 
\[
 (\Phi^s, \Phi^t) = (\hat F, \hat F)(\hat V_+, \hat V_-).
\]
 Again note that $F^{\mu} =  \hat F \ell+  {\hat F}^* \bar \ell$
 and accordingly the Maurer-Cartan form $\alpha^{\mu}$ of $F^{\mu}$
 taking values in $\lisu$
 can be translated to $\alpha^{\mu} = \hat \alpha \ell + \hat \alpha^{*} \bar \ell$, 
 where 
\begin{equation}\label{eq:alphalambda2}
\hat  \alpha = \hat U ds + \hat V dt\quad \mbox{with}\quad
 \partial_s \hat F = \hat F \hat U, 
 \quad  
 \partial_t \hat F = \hat F \hat V.
\end{equation}
 Note that $ \hat \alpha = \hat \alpha(\mu )$ 
 and $\hat \alpha^{*} = - \sigma_3 \overline{\hat \alpha(1 /\bar \mu)}^T \sigma_3$.
 Then a straightforward computation shows that 
\begin{align}\label{eq:hatUV}
 \hat U =  
\begin{pmatrix}
 \frac12 (\log \hat h)_{s} &  - \frac14\lambda^{-1}  \hat h \\  
  \lambda^{-1} Q (s)\hat h^{-1}& - \frac12 (\log \hat h)_{s}
\end{pmatrix},\quad 
\hat V  = 
\begin{pmatrix}
 -\frac12(\log \hat h)_{t} & -\lambda R(t)\hat h^{-1}\\ 
 \frac14\lambda \hat h & \frac12(\log \hat h)_t
\end{pmatrix}
\end{align}
 hold, where for a angle function $h= h(z, \bar z)$, $\hat h$ is defined by 
 $\hat h(s, t) = \Re h(s, t) + \Im h(s, t)$, and $F^{\mu}$ has the Maurer-Caran form 
 in \eqref{eq:UVmu}.

Note that  when we consider that $\alpha$ takes values in $\isu$, the spectral parameter takes 
\[
\mu = e^{\ip \theta}  = \cosh(\theta) + \ip \sinh(\theta) \in 
 \S_1^1\quad 
 ( \theta \in \R).
\]
 Then the corresponding spectral parameter $\lambda$ is given by
\[
 \lambda = e^{\theta}= \cosh(\theta) + \sinh(\theta) \in \R^+. 
\]
 We would like to note that,
 in \cite{DIT}, the null coordinate is used.
 Moreover, the spectral parameter $\lambda$ is replaced by $\lambda^{-1}$,
 and then $\hat U$ (resp. $\hat V$) in this paper
 plays a role of
 $U(\lambda^{-1})$ (resp. $V(\lambda^{-1})$) in \cite[Section 5]{DIT}.

 \textbf{Acknowledgement:}  We would  like to thank the anonymous referee for 
 carefully reading the manuscript and for pointing out to 
 us a number of typographical errors and for giving good suggestions.

\bibliographystyle{plain}
\def\cprime{$'$}

\end{document}